\documentclass[11pt,letterpaper]{article}

\usepackage{graphicx,subfigure}

\def\jamesmode{0}
\def\arxivmode{0}
\def\fastmode{0}

\def\showauthornotes{0}

\def\showkeys{0}
\def\showdraftbox{1}
\def\showcolorlinks{1}
\def\usemicrotype{1}
\def\showfixme{0}

\newcommand{\children}{\Lambda}

\newcommand{\level}{\mathrm{lev}}
\newcommand{\levelt}{\mathrm{lev}^{*}}

\newcommand{\dloc}{\vvmathbb{d}_{\mathrm{loc}}}
\newcommand{\graphs}{\cG}
\newcommand{\rgraphs}{\cG_{\bullet}}
\newcommand{\rrgraphs}{\cG_{\bullet \bullet}}
\newcommand{\cgraphs}{\graphs^*}
\newcommand{\crgraphs}{\rgraphs^*}
\newcommand{\crrgraphs}{\rrgraphs^*}

\newcommand{\todl}{\Rightarrow}

\newcommand{\equivclass}[1]{[#1]}

\newcommand{\vvS}{\vvmathbb{S}}

\newcommand{\avgd}{\bar{d}}
\newcommand{\dimspec}{\mathrm{dim}_{\mathsf{sp}}}
\newcommand{\dimspecover}{\obar{\mathrm{dim}}_{\mathsf{sp}}}

\newcommand{\dimconf}{\mathrm{dim}_{\mathsf{cg}}}
\newcommand{\dimconfunder}{\ubar{\mathrm{dim}}_{\mathsf{cg}}}
\newcommand{\dimconfover}{\obar{\mathrm{dim}}_{\mathsf{cg}}}

\ifnum\arxivmode=0  \else  \fi

\ifnum\fastmode=0
\usepackage{etex}
\fi

\ifnum\fastmode=0
\usepackage[l2tabu, orthodox]{nag}
\fi

\usepackage{xspace,enumerate}

\usepackage[dvipsnames]{xcolor}

\ifnum\fastmode=0
\usepackage[T1]{fontenc}
\usepackage[full]{textcomp}
\fi

\usepackage[american]{babel}

\usepackage{mathtools}

\usepackage{amsthm}

\newtheorem{theorem}{Theorem}[section]
\newtheorem*{theorem*}{Theorem}

\newtheorem{proposition}[theorem]{Proposition}
\newtheorem*{proposition*}{Proposition}
\newtheorem{lemma}[theorem]{Lemma}
\newtheorem*{lemma*}{Lemma}
\newtheorem{corollary}[theorem]{Corollary}
\newtheorem*{conjecture*}{Conjecture}
\newtheorem{fact}[theorem]{Fact}
\newtheorem*{fact*}{Fact}

\newtheorem*{exercise*}{Exercise}

\newtheorem*{hypothesis*}{Hypothesis}

\theoremstyle{definition}

\newtheorem{example}[theorem]{Example}

\newtheorem{exercise-easy}[theorem]{Exercise}
\newtheorem{exercise-med}[theorem]{Exercise}
\newtheorem{exercise-hard}[theorem]{Exercise$^\star$}
\newtheorem{claim}[theorem]{Claim}
\newtheorem*{claim*}{Claim}

\newtheorem{remark}[theorem]{Remark}
\newtheorem*{remark*}{Remark}

\newtheorem*{observation*}{Observation}

\usepackage[
letterpaper,
top=0.7in,
bottom=0.9in,
left=1in,
right=1in]{geometry}

\ifnum\arxivmode=0
\usepackage{newpxtext} %
\usepackage{textcomp} %
\usepackage[varg,bigdelims]{newpxmath}
\usepackage[scr=rsfso]{mathalfa}%
\usepackage{bm} %
\let\mathbb\varmathbb
\fi

\ifnum\arxivmode=1
\usepackage[varg]{pxfonts} %
\usepackage{textcomp} %
\usepackage[scr=rsfso]{mathalfa}%
\usepackage{bm} %
\fi

\ifnum\showkeys=1
\usepackage[color]{showkeys}
\fi

\makeatletter
\@ifpackageloaded{hyperref}
{}
{\usepackage[pagebackref]{hyperref}}

\definecolor{bleudefrance}{rgb}{0.01, 0.1, 1.0}
\definecolor{azure}{rgb}{0.0, 0.5, 1.0}

\ifnum\showcolorlinks=1
\hypersetup{
pagebackref=true,
colorlinks=true,
urlcolor=blue,
linkcolor=blue,
citecolor=OliveGreen}
\fi

\ifnum\showcolorlinks=0
\hypersetup{
pagebackref=true,
colorlinks=false,
pdfborder={0 0 0}}
\fi

\usepackage{prettyref}

\newcommand{\savehyperref}[2]{\texorpdfstring{\hyperref[#1]{#2}}{#2}}

\newrefformat{eq}{\savehyperref{#1}{\textup{(\ref*{#1})}}}
\newrefformat{lem}{\savehyperref{#1}{Lemma~\ref*{#1}}}
\newrefformat{def}{\savehyperref{#1}{Definition~\ref*{#1}}}
\newrefformat{thm}{\savehyperref{#1}{Theorem~\ref*{#1}}}
\newrefformat{cor}{\savehyperref{#1}{Corollary~\ref*{#1}}}
\newrefformat{cha}{\savehyperref{#1}{Chapter~\ref*{#1}}}
\newrefformat{sec}{\savehyperref{#1}{Section~\ref*{#1}}}
\newrefformat{app}{\savehyperref{#1}{Appendix~\ref*{#1}}}
\newrefformat{tab}{\savehyperref{#1}{Table~\ref*{#1}}}
\newrefformat{fig}{\savehyperref{#1}{Figure~\ref*{#1}}}
\newrefformat{hyp}{\savehyperref{#1}{Hypothesis~\ref*{#1}}}
\newrefformat{alg}{\savehyperref{#1}{Algorithm~\ref*{#1}}}
\newrefformat{rem}{\savehyperref{#1}{Remark~\ref*{#1}}}
\newrefformat{item}{\savehyperref{#1}{Item~\ref*{#1}}}
\newrefformat{step}{\savehyperref{#1}{step~\ref*{#1}}}
\newrefformat{conj}{\savehyperref{#1}{Conjecture~\ref*{#1}}}
\newrefformat{fact}{\savehyperref{#1}{Fact~\ref*{#1}}}
\newrefformat{prop}{\savehyperref{#1}{Proposition~\ref*{#1}}}
\newrefformat{prob}{\savehyperref{#1}{Problem~\ref*{#1}}}
\newrefformat{claim}{\savehyperref{#1}{Claim~\ref*{#1}}}
\newrefformat{relax}{\savehyperref{#1}{Relaxation~\ref*{#1}}}
\newrefformat{red}{\savehyperref{#1}{Reduction~\ref*{#1}}}
\newrefformat{part}{\savehyperref{#1}{Part~\ref*{#1}}}
\newrefformat{ex}{\savehyperref{#1}{Exercise~\ref*{#1}}}
\newrefformat{property}{\savehyperref{#1}{Property~\ref*{#1}}}
\newrefformat{assume}{\savehyperref{#1}{Assumption~\ref*{#1}}}

\newcommand{\Sref}[1]{\hyperref[#1]{\S\ref*{#1}}}

\usepackage{nicefrac}

\ifnum\fastmode=0
\ifnum\usemicrotype=1
\usepackage{microtype}
\fi
\fi

\ifnum\showauthornotes=2
\newcommand{\mynotes}[1]{{\sffamily\small\color{teal}{#1}}\medskip}
\newcommand{\Authornote}[2]{{\sffamily\small\color{blue}{[#1: #2]}}\medskip}
\newcommand{\Authornotecolored}[3]{{\sffamily\small\color{#1}{[#2: #3]}}}
\newcommand{\Authorcomment}[2]{{\sffamily\small\color{gray}{[#1: #2]}}}
\newcommand{\Authorstartcomment}[1]{\sffamily\small\color{gray}[#1: }

\newcommand{\Authorfnote}[2]{\footnote{\color{red}{#1: #2}}}
\newcommand{\Authorfixme}[1]{\Authornote{#1}{\textbf{??}}}
\newcommand{\Authormarginmark}[1]{\marginpar{\textcolor{red}{\fbox{\Large #1:!}}}}
\newcommand{\myexplain}[1]{{\sffamily\small\color{red}{\noindent [Explanation:\medskip\newline \begin{quote}#1\hfill]\end{quote}}}\medskip}

\else

\newcommand{\mynotes}[1]{}
\newcommand{\Authornote}[2]{}
\newcommand{\Authornotecolored}[3]{}
\newcommand{\Authorcomment}[2]{}
\newcommand{\Authorstartcomment}[1]{}

\newcommand{\Authorfnote}[2]{}
\newcommand{\Authorfixme}[1]{}
\newcommand{\Authormarginmark}[1]{}
\newcommand{\myexplain}[1]{}

\ifnum\showauthornotes=1
\renewcommand{\Authornote}[2]{{\sffamily\small\color{blue}{[#1: #2]}}\medskip}
\renewcommand{\Authornotecolored}[3]{{\sffamily\small\color{#1}{[#2: #3]}}}
\fi

\fi

\ifnum\showfixme=0

\fi

\usepackage{boxedminipage}

\newcommand{\Esymb}{\mathbb{E}}
\newcommand{\Psymb}{\mathbb{P}}

\DeclareMathOperator*{\E}{\Esymb}

\DeclareMathOperator*{\ProbOp}{\Psymb}

\renewcommand{\Pr}{\ProbOp}

\newcommand{\textparen}[1]{\text{(#1)}}

\ifx\because\undefined
\newcommand{\because}[1]{\textparen{because #1}}
\else
\renewcommand{\because}[1]{\textparen{because #1}}
\fi

\newcommand{\seteq}{\mathrel{\mathop:}=}

\newcommand{\bigmid}{~\big|~}

\newcommand\bdot\bullet

\ifx\mathds\undefined %
\newcommand{\Ind}{\mathbb I}
\else
\newcommand{\Ind}{\mathds 1}
\fi

\DeclareMathOperator{\vol}{vol}

\newcommand{\Z}{\mathbb Z}
\newcommand{\N}{\mathbb N}
\newcommand{\R}{\mathbb R}

\newcommand{\cB}{\mathcal B}
\newcommand{\cC}{\mathcal C}
\newcommand{\cD}{\mathcal D}
\newcommand{\cE}{\mathcal E}

\newcommand{\cG}{\mathcal G}

\newcommand{\cR}{\mathcal R}

\newcommand{\cV}{\mathcal V}

\newcommand{\cX}{\mathcal X}

\renewcommand{\leq}{\leqslant}

\renewcommand{\geq}{\geqslant}

\ifnum\showdraftbox=1

\else

\fi

\let\epsilon=\varepsilon

\numberwithin{equation}{section}

\newcommand\MYcurrentlabel{xxx}

\newcommand{\MYstore}[2]{%
  \global\expandafter \def \csname MYMEMORY #1 \endcsname{#2}%
}

\newcommand{\MYload}[1]{%
  \csname MYMEMORY #1 \endcsname%
}

\newcommand{\MYnewlabel}[1]{%
  \renewcommand\MYcurrentlabel{#1}%
  \MYoldlabel{#1}%
}

\newcommand{\MYdummylabel}[1]{}

\newcommand{\torestate}[1]{%
  \let\MYoldlabel\label%
  \let\label\MYnewlabel%
  #1%
  \MYstore{\MYcurrentlabel}{#1}%
  \let\label\MYoldlabel%
}

\newcommand{\restatetheorem}[1]{%
  \let\MYoldlabel\label
  \let\label\MYdummylabel
  \begin{theorem*}[Restatement of \prettyref{#1}]
    \MYload{#1}
  \end{theorem*}
  \let\label\MYoldlabel
}

\newcommand{\restatelemma}[1]{%
  \let\MYoldlabel\label
  \let\label\MYdummylabel
  \begin{lemma*}[Restatement of \prettyref{#1}]
    \MYload{#1}
  \end{lemma*}
  \let\label\MYoldlabel
}

\newcommand{\restateprop}[1]{%
  \let\MYoldlabel\label
  \let\label\MYdummylabel
  \begin{proposition*}[Restatement of \prettyref{#1}]
    \MYload{#1}
  \end{proposition*}
  \let\label\MYoldlabel
}

\newcommand{\restatefact}[1]{%
  \let\MYoldlabel\label
  \let\label\MYdummylabel
  \begin{fact*}[Restatement of \prettyref{#1}]
    \MYload{#1}
  \end{fact*}
  \let\label\MYoldlabel
}

\newcommand{\restate}[1]{%
  \let\MYoldlabel\label
  \let\label\MYdummylabel
  \MYload{#1}
  \let\label\MYoldlabel
}

\newcommand{\addreferencesection}{
  \phantomsection
\ifnum\stocmode=0
  \addcontentsline{toc}{section}{References}
\else
  \addcontentsline{toc}{section}{References \hspace*{1in} --------- End of extended abstract ---------}
\fi

}

\newcommand{\e}{\epsilon}

\let\origparagraph\paragraph
\renewcommand{\paragraph}[1]{\vspace*{-20pt}\hspace*{-5pt}\origparagraph{#1.}}

\allowdisplaybreaks

\usepackage{paralist}

\usepackage{comment}

\let\pref=\prettyref

\newcommand{\dist}{\mathsf{dist}}

\newcommand{\diam}{\mathsf{diam}}

\renewcommand{\Ind}{\vvmathbb{1}}
\ifnum\jamesmode=0
\fi
\let\1\Ind

\newcommand\f{\varphi}

\usepackage{accents}

\usepackage{enumitem}
\usepackage[normalem]{ulem}

\newcommand\myuline{\bgroup\markoverwith{\rule[-0.4ex]{2pt}{0.2mm}}\ULon}
\newcommand{\obar}[1]{\smash{\mkern3mu\overline{\mkern-3mu \vphantom{\scalebox{0.85}{\ensuremath{#1}}} \smash{#1}\mkern-3mu}\mkern3mu}}
\newcommand{\ubar}[1]{\smash{\mkern2mu\myuline{\mkern-2mu \smash{#1}\mkern-2mu}\mkern2mu}}

\newlist{assumptions}{enumerate}{10}
\setlist[assumptions]{label*=\arabic*}

\renewcommand{\deg}{\mathrm{deg}}

\newcommand{\dmax}{d_{\max}}
\newcommand{\len}{\mathrm{len}}

\renewcommand{\Z}{\vvmathbb{Z}}
\renewcommand{\R}{\vvmathbb{R}}

\begin{document}

\title{Discrete uniformizing metrics \\ on distributional limits of sphere packings}

   \author{James R. Lee \\ {\small University of Washington}}
\date{}

\maketitle

\begin{abstract}
   Suppose that $\{G_n\}$ is a sequence of finite graphs
   such that each $G_n$ is the tangency graph of a sphere packing in $\R^d$.
   Let $\rho_n$ be a uniformly random vertex of $G_n$ and suppose that
   $(G,\rho)$ is the distributional limit of $\{(G_n,\rho_n)\}$ in the
   sense of Benjamini and Schramm.
   Then the conformal growth exponent of $(G,\rho)$ is at most $d$.
   In other words, there exists a unimodular ``unit volume'' weighting of the graph metric
   on $(G,\rho)$ such that the volume growth of balls in the weighted path metric
   is bounded by a polynomial of degree $d$.
   This assertion generalizes to limits of graphs that can be ``quasi-packed'' in an Ahlfors $d$-regular
   metric measure space.

   It implies that,
   under moment conditions on the degree of the root $\rho$, 
   the almost sure spectral dimension
   of $G$ is at most $d$.
   This fact was known previously only for graphs packed in $\R^2$
   (planar graphs), and the case $d > 2$ eluded approaches
   based on extremal length.
   In the process of bounding the spectral dimension, we
   establish that the spectral measure of $(G,\rho)$
   is dominated by a variant of the $d$-dimensional Weyl law.
\end{abstract}

\begingroup
\hypersetup{linktocpage=false}
\tableofcontents
\endgroup

\section{Introduction}

The theory of random planar graphs has been an active
area of study in the last twenty years (see, e.g., \cite{BenjaminiICM}), inspired
partially by the connection to two-dimensional quantum gravity \cite{ADJ97}.
As noted by Benjamini and Curien \cite{BC11}, an analogous theory
in higher dimensions has proved elusive, in part based on the difficulty
of enumeration for higher-dimenisonal simplicial complexes (see \cite{BZ11}
and the references therein).

To address this discrepancy, the authors of \cite{BC11} explored
the extension of analytic and probabilistic methods based on potential theory.
A graph $G$ is said to be {\em sphere-packed in $\R^d$} if $G$
is the tangency graph of a collection of
interior-disjoint spheres in $\R^d$.
Benjamini and Curien proved that if a family of finite graphs can be sphere-packed in $\R^d$
with spheres of bounded aspect ratio (so that the ratio of the radii of tangent
spheres is $O(1)$), then a distributional limit of such graphs is
{\em $d$-parabolic.}

Roughly speaking, $d$-parabolicity means that the $\ell_d$ extremal
length from a fixed vertex to $\infty$ is infinite, where
the $\ell_d$ extremal length is a natural analog 
Cannon's vertex extremal length \cite{Cannon94} (the case $d=2$); see also 
\cite{Duffin62} and \pref{sec:vel}.
 It is well-known that the special case of $2$-parabolicity carries strong
probabilistic significance; for instance, for graphs with uniformly bounded degrees,
$2$-parabolicity is equivalent to recurrence of the random walk (see \cite{Duffin62,DS84}).
For $d > 2$, the theory of $\ell_d$ extremal length seems
somewhat less powerful, and is not known to yield such
control on the random walk.

In this work, we study a related notion
that one might refer to as the ``extremal growth rate\,.''
For graphs that can be sphere-packed in $\R^d$,
we show that it is possible to construct metrics that uniformize
their underlying geometry
so that the {\em counting measure} has $d$-dimensional
volume growth.
Employing the results of \cite{Lee17a},
one does obtain substantial probabilistic consequences,
including $d$-dimensional lower bounds on the diagonal heat kernel
(see \pref{thm:heat-kernel} below).
Moreover, our results hold in considerable generality;
they require no assumption on the ratio of radii of adjacent
balls in the packing,
and they extend to graphs that can be ``quasi-packed''
in an Ahlfors regular metric measure space, as we now describe.

\medskip
\noindent
{\bf Quasi-packings and the spectral dimension.}
Consider a metric space $(X,\dist)$.
A {\em $\tau$-quasi-ball in $X$} is a Borel set $S \subseteq X$
that is sandwiched between two closed balls:
$B(x,r) \subseteq S \subseteq B(x,\tau r)$
for some $x \in X, r > 0$.
Let $\cB_{\tau}$ denote the collection of $\tau$-quasi-balls in $X$.
Say that a graph $G$ is {\em $(\tau,M)$-quasi-packed in $(X,\dist)$}
if there is a mapping $\Phi : V(G) \to \cB_{\tau}$ 
that satisfies:
\begin{enumerate}
   \item {\bf Quasi-tangency:}
      \begin{equation}\label{eq:coarse-tan}
         \{u,v\} \in E(G) \implies \dist(\Phi(u),\Phi(v)) \leq \tau \min \left\{\diam(\Phi(u)),\diam(\Phi(v))\right\}.
      \end{equation}
   \item {\bf Quasi-multiplicity:}  For every $x \in X$ and $r \geq 0$:
      \begin{equation}\label{eq:bdd-mult}
         \# \left\{\vphantom{\bigoplus} v \in V(G) : B(x, r) \cap \Phi(v) \neq \emptyset \textrm{ and } \diam(\Phi(v)) \geq \tau r \right\} \leq M\,.
      \end{equation}
\end{enumerate}
Say that a graph $G$ {\em quasi-packs in $(X,\dist)$} if $G$ is $(\tau,M)$-quasi-packed in $(X,\dist)$
for some numbers $M,\tau~\geq~1$.
A family $\{G_n\}$ of graphs {\em uniformly quasi-packs in $(X,\dist)$}
if there are $M,\tau \geq 1$ such that each $G_n$ is $(\tau,M)$-quasi-packed in $(X,\dist)$.
Of course, the collection $\{ \Phi(v) : v \in V(G)\}$ is only a genuine packing for $M=1$.
We now state a representative theorem.

\begin{theorem}
   Consider a sequence of random rooted finite graphs $\{(G_n,\rho_n)\}$ with $\rho_n \in V(G_n)$
   chosen uniformly at random.
   Suppose the family $\{G_n\}$ has uniformly bounded degrees and is uniformly quasi-packed in
   an Ahlfors $d$-regular metric measure space.  If $(G,\rho)$ is the distributional limit of
   this sequence, then almost surely $\dimspecover(G) \leq d$.
   Moreover, if $d=2$, then $G$ is almost surely recurrent.
\end{theorem}

Here, ``distributional limit'' refers to convergence in the Benjamini-Schramm sense
(i.e., in the weak local topology):
$\{(G_n,\rho_n)\} \to (G,\rho)$ means that the laws
of neighborhoods of $\rho_n$ in $G_n$ converge
to the law of neigborhoods of $\rho$ in $G$,
where neighborhoods are considered up to rooted isomorphism.
See \pref{sec:unimodular} for precise definitions.

And we use $\dimspecover$ to denote the {\em upper spectral dimension:}
\[
   \dimspecover(G) \seteq \limsup_{n \to \infty} \frac{-2 \log p_{2n}^G(v,v)}{\log n}\,,
\]
where $p_t^G(v,v) = \Pr[X_t = v \mid X_0 = v]$ and $\{X_t\}$ is the standard random walk on $G$.
(The value does not depend on the choice of $v \in V(G)$.)

\begin{remark}[Coarse packings]
   It is not hard to check that
   if two metric spaces $X$ and $Y$ are bi-Lipschitz
   equivalent, then $G$ quasi-packs in $X$ if and only if $G$ 
   quasi-packs in $Y$, making the notion a bi-Lipschitz invariant.
   More generally, it is a quasisymmetric invariant
   when $X$ is uniformly perfect.
   See \pref{sec:stability}.

To relate quasi-packings to more standard notions,
it is helpful to consider a simpler set of assumptions.
Say that a graph $G$ {\em coarsely packs in $X$} if there
are numbers $M,\tau \geq 1$
and a map $\Phi : V(G) \to \cB_{1}$
so that \eqref{eq:coarse-tan} is satisfied,
as well as
\begin{equation}\label{eq:quasi-mult}
   \# \{ v \in V(G) : x \in \Phi(v) \} \leq M\quad \forall x \in X\,.
\end{equation}
Note that this is simply \eqref{eq:bdd-mult} for $r=0$
and $\cB_1$ is precisely the collection of closed balls in $X$.
If $(X,\dist)$ is an Ahlfors $d$-regular length space
(cf. \pref{sec:metric-spaces}) and $G$ coarsely packs in $X$,
then it quasi-packs in $X$.
This is proved in \pref{sec:round-bodies}.

This implies that if $G$ is the tangency graph of interior-disjoint spheres in $\R^d$,
then it is automatically $(\tau,M)$-quasi-packed in $\R^d$ for some $M,\tau \geq 1$ depending
only on $d$.
For a non-Euclidean example,
consider that the same is true of 
the tangency graphs of interior-disjoint balls
in the Heisenberg groups equipped with their 
Carnot-Carath{\'e}odory metrics.
See \pref{sec:round-bodies} for a detailed discussion.
In general, the reader will suffer no
great conceptual loss by thinking only of classical sphere packings in $\R^d$.
\end{remark}

\newcommand\myov{\bgroup\markoverwith{\rule[8.5pt]{0.1pt}{0.5pt}}\ULon}
\newcommand{\myubar}[1]{\mkern2mu\myov{\mkern-2mu #1\mkern-2mu}\mkern2mu}

\subsection{Discrete conformal metrics on sphere-packed graphs}

   Consider a locally finite, connected graph $G$.
A {\em conformal metric} (or {\em conformal weight}) on $G$ is a map $\omega : V(G) \to \R_+$.
This endows $G$ with a graph distance as follows:  Give to every edge $\{u,v\} \in E(G)$ a length
$\len_{\omega}(\{u,v\}) \seteq \frac12 (\omega(u)+\omega(v))$.
This prescribes to every path $\gamma = \{v_0, v_1, v_2, \ldots\}$ in $G$ the induced length
\[
   \len_{\omega}(\gamma) \seteq \sum_{k \geq 0} \len_{\omega}(\{v_k, v_{k+1}\})\,.
\]
Now for $u,v \in V(G)$,
one defines the path metric $\dist_{\omega}(u,v)$ as the infimum of the lengths of all $u$-$v$ paths in $G$.
Denote the closed ball
\[
   B_{\omega}(x,R) \seteq \left\{ y \in V(G) : \dist_{\omega}(x,y) \leq R \right\}\,.
\]
We can now state a special case of our main technical theorem;
the connection to distributional limits 
and random walks is discussed subsequently.

\begin{theorem}\label{thm:finite-sphere-packing-intro-0}
   For every $d,M,\tau \geq 1$ and every
   Ahlfors $d$-regular metric measure space $\cX$
   there is a constant $C$ such that the following holds.
   If $G=(V,E)$ is a finite graph that is $(\tau,M)$-quasi-packed in $\cX$, then
   there is a conformal metric $\omega : V \to \R_+$
   that satisfies
   \[
      \frac{1}{|V|} \sum_{x \in V} \omega(x)^{d} = 1\,,
   \]
   and 
   such that \[\max_{x \in V(G)} |B_{\omega}(x,R)| \leq C R^d (\log R)^{2}\quad \forall R \geq 1\,.\]
\end{theorem}

The method of proof is based partially on a celebrated lemma
of Benjamini and Schramm \cite{BS01}.  They show that if $\{G_n\}$
is a sequence of finite planar triangulations with uniformly bounded degrees
and $\{G_n\}$ converges to a distributional limit $(G,\rho)$,
then almost surely any circle packing of $G$ has at most one
accumulation point in the plane. 
An analogous result holds for graphs sphere-packed in $\R^d$ when $d > 2$ \cite{BC11}.

We argue that, in a quantative sense, as long
as the accumulation points remain separated, one can construct
a multi-scale reweighting of the spheres
in the packing, endowing the graph with a metric that reflects its $d$-dimensional
structure with respect to the underlying counting measure.
This is carried out in \pref{sec:sphere-packings}.

\subsection{Conformal growth exponents}

If $(G,\rho)$ is random rooted graph, then a {\em conformal metric on $(G,\rho)$}
is a random triple $(G',\omega,\rho')$ with $\omega : V(G) \to \R_+$
such that $(G,\rho)$ and $(G',\rho')$
have the same law.  We say that the conformal weight is {\em normalized} if $\E\left[\omega(\rho)^2\right] = 1$.
One thinks of such a metric $\omega : V(G) \to \R_+$ as deforming the geometry
of the underlying graph subject to a bound
on the total ``area.''
As shown in \cite{Lee17a}, 
normalized conformal metrics
with nice geometric properties form a powerful tool in understanding
the spectral geometry of $(G,\rho)$.

In the present work, we consider {\em unimodular} random graphs (see \pref{sec:unimodular});
such graphs arise naturally as distributional limits of finite random rooted graphs $\{(G_n,\rho_n)\}$
where $\rho_n \in V(G_n)$ is chosen uniformly at random.
We will consider only unimodular conformal metrics $\omega$ on $(G,\rho)$;
in other words, the setting where $(G,\omega,\rho)$ is unimodular as a marked network
in the sense of \cite{aldous-lyons}.

\medskip
\noindent
{\bf Conformal growth exponents.}
Consider a unimodular random graph $(G,\rho)$.
In \cite{Lee17a}, we defined the {\em upper and lower conformal growth exponents of $(G,\rho)$}, respectively, by
\begin{align}
   \label{eq:overdef}
   \dimconfover(G,\rho) & \seteq \inf_{\omega} \limsup_{R \to \infty} \frac{\log \|\#B_{\omega}(\rho,R)\|_{L^\infty}}{\log R}\,, \\
   \label{eq:underdef}
   \dimconfunder(G,\rho) & \seteq \inf_{\omega} \liminf_{R \to \infty} \frac{\log \|\#B_{\omega}(\rho,R)\|_{L^\infty}}{\log R}\,,
\end{align}
where the infimum is over all normalized unimodular conformal metrics on $(G,\rho)$, and 
we use
$\|X\|_{L^{\infty}}$ to denote the essential supremum of a random variable $X$,
and $\# S$ to denote the cardinality of a set $S$.

When $\dimconfover(G,\rho) = \dimconfunder(G,\rho)$, define the {\em conformal growth exponent}
by
\[
   \dimconf(G,\rho) \seteq \dimconfover(G,\rho) = \dimconfunder(G,\rho)\,.
\]
Note that the quantities $\dimconfover,\dimconfunder,\dimconf$ are functions of the law of $(G,\rho)$;
they are not defined on (fixed) rooted graphs.

The conformal growth exponent bears a philosophical
resemblance to Pansu's notion of {\em conformal dimension} \cite{Pansu89}.
The relationship between sphere packings in $\R^2$ and conformal mappings
is classical and well-understood.
For an emerging more general theory,
we refer to Pansu's recent work \cite{Pansu16}
which explores in detail the relationship between
sphere packings and the theory of large-scale conformal maps.

\medskip
\noindent
{\bf $L^q$ conformal growth rate.}
Let us define a generalization:  If $(G,\omega,\rho)$ is a unimodular random conformal graph,
we denote
\[
   \|\omega\|_{L^q} \seteq \left(\E \omega(\rho)^q\right)^{1/q}\,.
\]
Say that $\omega$ is {\em $L^q$-normalized} if $\|\omega\|_{L^q}=1$.

Define the analogous $L^q$ quantities:  $\dimconfover^q, \dimconfunder^q, \dimconf^q$
where now the infima in \eqref{eq:overdef} and \eqref{eq:underdef} are over all $L^q$-normalized
conformal metrics on $(G,\rho)$.
Observe that, by monotonicity of $L^q$ norms, we have
\[
   q \leq q' \implies \dimconf^q(G,\rho) \leq \dimconf^{q'}(G,\rho)\,.
\]

The next theorem constitutes the main new technical theorem presented here.
We use $\todl$ to denote convergence
in the distributional sense; see \pref{sec:unimodular}.

\begin{theorem}\label{thm:Rd-packings}
   For any $d\geq 2$, the following holds.
   If $(G,\rho)$ is the distributional limit of finite graphs that are
   uniformly quasi-packed in an Ahlfors $d$-regular metric measure space,
   then there is an $L^{d}$-normalized unimodular conformal metric $\omega : V(G) \to \R_+$ such that
   almost surely, 
   for all $R \geq 1$,
   \begin{equation}\label{eq:rd-packings}
      \left|B_{\omega}(\rho,R)\right| \leq O(R^{d} (\log R)^{2})\,.
   \end{equation}
   In particular, $\dimconfover(G,\rho) \leq d$.
\end{theorem}
The last assertion follows from $\dimconfover(G,\rho) = \dimconfover^2(G,\rho) \leq \dimconfover^d(G,\rho)$.
If $\cX$ is Ahlfors $d$-regular with $d < 2$, the conclusion $\dimconfover(G,\rho) \leq 2$ still holds;
see \pref{sec:sphere-packings}.
We remark that some $(\log R)^{O(1)}$ factor is necessary
even for the case of planar graphs; see \cite[\S 2]{Lee17a}.

\medskip

A primary motivation for \pref{thm:Rd-packings} is that such metrics
can be used to obtain estimates on the heat kernel and spectral measure of $G$.
  For a locally finite, connected graph $G$, denote the discrete-time
      heat kernel
      \[p^G_T(x,y) \seteq \Pr[X_T=y \mid X_0=x], \qquad x,y \in V(G)\,,\]
      where $\{X_n\}$ is the standard random walk on $G$ and $T \in \N$.
      We recall the {\em spectral dimension of $G$:}
      \[
         \dimspec(G) \seteq \lim_{n \to \infty} \frac{-2 \log p^G_{2n}(x,x)}{\log n},
      \]
      whenever the limit exists.  If the limit does exist, then it is the same for all $x \in V(G)$.

      Say that a real-valued random variable $X$ has {\em negligible tails}
      if its tails decay faster than any inverse polynomial:
      \begin{equation}\label{eq:neg-tails-def}
         \lim_{n \to \infty} \frac{\log n}{|\!\log \Pr[|X| > n]|} = 0\,,
      \end{equation}
      where we take $\log(0)=-\infty$ in the preceding definition
      (in the case that $X$ is essentially bounded).
The next theorem is from \cite{Lee17a};
it asserts that if $\dimconfover(G,\rho) \leq d$, then 
almost surely $G$ admits $d$-dimensional lower bounds on the diagonal heat kernel:
\[
   p^G_{2n}(\rho,\rho) \geq n^{-d/2-o(1)} \quad \textrm{as} \quad n \to \infty\,.
\]

\begin{theorem}\label{thm:intro-dimspec1}
   Suppose that $(G,\rho)$ is a unimodular random graph such that $\deg_G(\rho)$
   has negligible tails.
   Then almost surely:
   \[
      \dimspecover(G) \leq \dimconfover(G,\rho)\,.
   \]
   In particular, if there is a number $d$ such that almost surely $\dimspec(G)=d$, then $d \leq \dimconfover(G,\rho)$.
\end{theorem}

In certain situations, one can give stronger estimates.
Indeed, when the conformal growth rate has only a polylogarithmic correction
as in \eqref{eq:rd-packings}, one obtains stronger results (see \cite[\S 4.2]{Lee17a}).

\begin{theorem}\label{thm:heat-kernel}
   Suppose $(G,\rho)$ is the distributional limit of finite graphs
   that are uniformly quasi-packed in an Ahlfors $d$-regular metric measure space $\cX$,
   and that $\deg_G(\rho)$ has exponential tails in the sense that
   \[
      \Pr[\deg_G(\rho) > k] \leq e^{-ck}
   \]
   for some $c > 0$.
   Then there is a constant $C \geq 1$ such that for $n$ sufficiently large,
   \[
      \Pr\left[p^G_{2n}(\rho,\rho) \geq \frac{n^{-d_*/2}}{(\log n)^C}\right] \geq 1 - \frac{1}{\log n}\,,
   \]
   where $d_* = \max(d,2)$.
\end{theorem}

\subsection{Gauged conformal growth and $d$-parabolicity}
\label{sec:vel}

\newcommand{\vel}{\mathsf{VEL}}

Consider a locally-finite connected graph $G=(V,E)$.
Let $\Gamma$ denote a collection of simple paths in $G$.
The {\em $\ell_d$-vertex extremal length of $\Gamma$} is defined by
\[
   \vel_d(\Gamma) \seteq \sup_{\omega} \inf_{\gamma \in \Gamma} \frac{\len_{\omega}(\gamma)}{\|\omega\|_{\ell_d(V)}}\,,
\]
where the infimum is over all conformal metrics on $G$, and
$\|\omega\|_{\ell_d(V)} \seteq \left(\sum_{v \in V} \omega(v)^d\right)^{1/d}$.

Fix a vertex $v_0 \in V$
and let $\Gamma_G(v_0)$ denote the set of infinite simple paths in $G$ emenating from $v_0$.
One says that $G$ is {\em $d$-parabolic} if $\vel_d(\Gamma_G(v_0))=\infty$ (see \cite{HS95,BS13}).  One can
check that this definition does not depend on the choice of $v_0 \in V$.

There are unimodular random graphs $(G,\rho)$ where $G$ is almost surely $d$-parabolic, but
$\dimconfunder^d(G,\rho) \geq \dimconfunder(G,\rho) = \infty$, and
other examples where $\dimconf(G,\rho)=d \geq 2$ but $G$ is almost surely not $d$-parabolic; see \pref{sec:parabolic}.

\medskip
\noindent
{\bf Gauged growth.}
On the other hand, there is a common strengthening of the conditions.
Say that $(G,\rho)$ has {\em $(C,R,d)$-growth} if
there is an $L^d$-normalized conformal metric $\omega : V(G)\to \R_+$ such that
\begin{equation}\label{eq:onescale-intro}
   \|\# B_{\omega}(\rho, R)\|_{L^{\infty}} \leq C R^d\,.
\end{equation}
Say that $(G,\rho)$ has {\em gauged $d$-dimensional conformal growth} if
there is a constant $C \geq 1$ such that $(G,\rho)$ has $(C,R,d)$-growth for all $R \geq 0$.
A sequence $\{(G_n,\rho_n)\}$ has {\em uniform gauged $d$-dimensional conformal growth}
if there is a constant $C \geq 1$ such that $(G_n,\rho_n)$ has $(C,R,d)$-growth for all $R \geq 0$ and $n \geq 1$.

It is straightforward to see that if $(G,\rho)$ has gauged $d$-dimensional growth, then
$\dimconfover^d(G,\rho) \leq d$:
For each $k \geq 1$, let $\omega_k$ denote an $L^d$-normalized conformal metric on $(G,\rho)$
satisfying \eqref{eq:onescale-intro} and define
\[
   \hat{\omega} \seteq \left(\frac{6}{\pi^2} \sum_{k \geq 1} \frac{\omega_{k}^d}{k^2}\right)^{1/d}.
\]
(By unimodularity of the triple $(G,\hat{\omega},\rho)$, it holds that almost surely $\sup_{x \in V(G)} \hat{\omega}(x) < \infty$;
see \pref{sec:unimodular}).

Establishing $d$-parabolicity is somewhat more involved;
the $d=2$ case of the following theorem is 
\cite[Thm. 2.1]{Lee17a}.  The general case is proved in \pref{sec:gauged-extremal}.

\begin{theorem}\label{thm:scaled-strong}
   For every $d \geq 1$, the following holds.
   If $(G,\rho)$ is a unimodular random graph such that $\deg_G(\rho)$ is essentially bounded
   and $(G,\rho)$ has gauged $d$-dimensional conformal growth, then $G$ is almost surely $d$-parabolic.
\end{theorem}

In order to establish \pref{thm:Rd-packings}, we prove the following stronger statement
in \pref{sec:sphere-packings}.

\begin{theorem}\label{thm:intro-scaled}
   For any $d\geq 1$, the following holds.
   If $(G,\rho)$ is the distributional limit of finite graphs that are
   uniformly quasi-packed in an Ahlfors $d$-regular metric measure space,
   then $(G,\rho)$ has gauged $\max(d,2)$-dimensional conformal growth.
\end{theorem}

Note that for the special case of planar graphs,
the conjunction of \pref{thm:scaled-strong} and \pref{thm:intro-scaled}
recovers the Benjamini-Schramm recurrence theorem \cite{BS01}
(that every distributional limit of finite planar graphs
with uniformly bounded degrees is almost surely $2$-parabolic).

\subsection{The spectral measure of $d$-dimensional graphs}

In order to obtain estimates like \pref{thm:intro-dimspec1} and \pref{thm:heat-kernel},
it is clear that one needs to control the moments of the spectral measure at the root.
Indeed, if $(G,\rho)$ is a random rooted graph,
then one can define the spectral measure $\mu \seteq \E[\mu_G^{\rho}]$, there $\mu_G^{v}$
is the unique probability measure on $\R$ such that for all integers $T \geq 1$:
\[
   \deg_G(v) \int \theta^T \,d\mu_G^v(\theta) = \langle \1_v, P_G^T \1_v\rangle_{\ell^2(G)}\,.
\]
Here, $P_G$ is the random walk operator on $G$ and $\ell^2(G)$ is the Hilbert space
of functions $f : V(G) \to \R$ with $\langle f,g\rangle_{\ell^2(G)} \seteq \sum_{x \in V(G)} \deg_G(x) f(x) g(x)$.
(See, e.g., \cite[\S 4.4.1]{Lee17a} and \cite[\S 1.4--1.5]{BSV17}.)
Note that $\mu$ is almost surely supported on $[-1,1]$.

In this formulation, one has:  For all integers $T \geq 1$,
\[
   \E\left[p^G_{2T}(\rho,\rho)\right] = \int \theta^{2T} d\mu(\theta)\,,
\]
hence an
elementary calculation shows that for every $d \geq 1$ and $T \geq 1$:
\[
   \frac14 \mu\!\left([1-\tfrac{1}{2T}, 1]\right) \leq
   \E\left[p^G_{2T}(\rho,\rho)\right] \leq T^{-d} + \mu\!\left(\left[1-\tfrac{d \log T}{2 T},1\right]\right)\,.
\]
Almost sure (quenched) lower bounds on $p^G_{2T}$ as in \pref{thm:intro-dimspec1}
are substantially more difficult to establish than lower bounds on $\E[p^G_{2T}(\rho,\rho)]$,
but annealed estimates are already interesting,
and one can draw a parallel to more classical settings.

\medskip
\noindent
{\bf The Weyl bound in $\R^d$.}
Consider a bounded domain $\Omega \subseteq \R^d$, and let
$\lambda_1 \leq \lambda_2 \leq \cdots$ be the
corresponding Neumann eigenvalues.  Let
$N_{\Omega}(\lambda) \seteq \# \{ k : \lambda_k \leq \lambda \}$
denote the eigenvalue counting function.
In 1912, addressing a conjecture of Lorentz, Weyl determined \cite{Weyl12}
the first-order asymptotics of $N_{\Omega}(\lambda)$ as $\lambda \to \infty$:
\[
   N_{\Omega}(\lambda) \sim c_d \vol(\Omega) \lambda^{d/2}\,,
\]
where $c_d$ is some constant depending only on the dimension.

In addressing a question of S. T. Yau on the spectrum of the Laplacian
on orientable surfaces, Korevaar \cite{Korevaar93} showed that
if $\Omega$ is a subdomain of a complete $d$-dimensional Riemannian manifold $(M, g_0)$
with nonnegative Ricci curvature, and $(M, \varphi g_0)$ is a finite-volume
conformal metric, then
\begin{equation}\label{eq:korevaar}
   N_{\Omega}(\lambda) \geq C_d \vol(\Omega, \varphi g_0) \lambda^{d/2}\,,
\end{equation}
where $C_d$ is a constant depending only on the dimension $d$.

Analogous results can be obtained for distributional limits of finite graphs that are sphere-packed $\R^d$.
Let $\nu$ denote the law of a random rooted graph $(G,\rho)$ and define $\bar{d}_{\nu} : [0,1] \to \R_+$ by
\[
   \bar{d}_{\nu}(\e) \seteq \sup \left\{ \E[\deg_G(\rho) \mid \cE] : \Pr(\cE) \geq \e \right\}\,,
\]
where the supremum is over all measurable sets $\cE$ with $\Pr(\cE) \geq \e$.

\begin{theorem}\label{thm:distributional-weyl}
   Consider $d \geq 1$ and an Ahlfors $d$-regular metric measure space $\cX$.
   Suppose $(G,\rho)$ is a distributional limit of finite graphs
   that are uniformly quasi-packed in $\cX$.
   Then there is a number $c > 0$ such that the following holds.
   Let $\nu$ denote the law of $(G,\rho)$, and let $\mu$ denote the corresponding spectral measure.
   For all $\e > 0$:
   \begin{equation}\label{eq:unimodular-weyl}
      \mu\!\left([1-\e,1]\right) \geq c \frac{(\log (1/\e))^{-2}}{\bar{d}_{\nu}(\e)} \e^{d/2}\,.
   \end{equation}
\end{theorem}
The asymptotic dependence on $\e$ is tight up to the $(\log(1/\e))^{-2}$ factor;
see \pref{rem:tight}.

\medskip
\noindent
{\bf The Laplacian spectrum of finite tangency graphs.}
\pref{thm:distributional-weyl} follows readily from an analogous result for finite graphs.
Let $G=(V,E)$ denote a finite connected graph with $n=|V|$.
Let $\{1-\lambda_k(G) : k=0,1,\ldots,n-1\}$ be
the eigenvalues of the random walk operator on $G$,
where \[0=\lambda_0(G) \leq \lambda_1(G) \leq \cdots \leq \lambda_{n-1}(G)\,.\]
Define the corresponding counting function:
\[
   N_G(\lambda) \seteq \# \{ k > 0 : \lambda_k(G) \leq \lambda \}\,. 
\]
Define also
\[
   \Delta_G(k) \seteq \max_{S \subseteq V : |S| \leq k} \sum_{x \in S} \deg_G(x)\,,
\]
where $\deg_G(x)$ denotes the degree of a vertex $x \in V$.
Note, in particular, that $\Delta_G(1)$ is the maximum degree in $G$.

Denote $\avgd_G(\e) \seteq \frac{\Delta_G(\e n)}{\e n}$.
In \cite{KLPT09}, addressing a conjecture of Spielman and Teng \cite{spielman-teng},
it is shown that there
is a constant $c > 0$ such that if $G$ is a planar graph, then for all $\lambda \in [0,1]$,
\begin{equation}\label{eq:klpt}
   N_G(\lambda) \geq \frac{c}{\Delta_G(1)} \lambda n\,.
\end{equation}
In \cite{Lee17a}, the author improves this bound to
\begin{equation}\label{eq:lee-counting}
   N_G(\lambda) \geq \frac{c}{\avgd_G(\lambda)} \lambda n\,,
\end{equation}
where $c > 0$ is some other constant.

While the utility of this improvement is not immediately apparent
in the finite setting, one should observe that \eqref{eq:klpt}
yields no information for a distributional limit $(G,\rho)$
in which there is no uniform bound on $\deg_G(\rho)$,
whereas \eqref{eq:lee-counting} yields \eqref{eq:unimodular-weyl}
in the case $d=2$ (and without the $\log(1/\e)^{-2}$ correction factor).
Moreover, the correct quantitative dependence is essential
to a spectral argument proving that the uniform infinite
planar triangulation is almost surely recurrent \cite{Lee17a}; this fact was first
established by Gurel-Gurevich and Nachmias \cite{GN13} using
effective resistance estimates.
In \pref{sec:spectral},
we use \pref{thm:finite-sphere-packing-intro-0} to establish
an analogous lower bound to \eqref{eq:korevaar}
for graphs sphere-packed in $\R^d$ (and their generalizations).

\begin{theorem}[Weyl bound for quasi-packed finite graphs]\label{thm:spectral-intro}
   For every $d,\tau,M \geq 1$ and every Ahlfors $d$-regular metric measure space $\cX$,
   there is a number $c > 0$ such that the following holds.
   If $G$ is an $n$-vertex graph that is $(\tau,M)$-quasi-packed in $\cX$, then for all $\lambda \in [0,1]$,
   \[
      N_{G}(\lambda) \geq \frac{c}{\avgd_G(\lambda)} \left(\log \frac{e}{\lambda}\right)^{-2} n \lambda^{d/2}\,.
   \]
\end{theorem}

\begin{remark}\label{rem:tight}
Up to the factor of $(\log (e/\lambda))^2$,
this bound is tight for a $d$-dimensional box $\{1,2,\ldots,n^{1/d}\}^d$ considered
as a subgraph of the integer lattice $\Z^d$.
Whether the $(\log (1/\lambda))^2$ factor can be removed from the
bound is an interesting open question.
\end{remark}

\subsection{Preliminaries}

We use the notations $\R_+ = [0,\infty)$ and $\Z_+ = \Z \cap \R_+$.

All graphs appearing in this paper are undirected and locally finite
and without loops or multiple edges.
If $G$ is such a graph, we use $V(G)$ and $E(G)$ to denote the
vertex and edge set of $G$, respectively.
If $S \subseteq V(G)$, we use $G[S]$ for the induced subgraph on $S$.
For $A,B \subseteq V(G)$, we write $E_G(A,B)$ for the set of edges
with one endpoint in $A$ and the other in $B$.
We write $\dist_G$ for the unweighted path metric on $V(G)$, and
$B_G(x,r) = \{ y \in V(G) : \dist_G(x,y) \leq r \}$
to denote the closed $r$-ball around $x \in V(G)$.
Also let $\deg_G(x)$ denote the degree of a vertex $x \in V(G)$, and $\dmax(G) = \sup_{x \in V(G)} \deg_G(x)$.

Write $G_1 \cong G_2$ to denote that $G_1$ and $G_2$ are isomorphic as graphs.
If $(G_1,\rho_1)$ and $(G_2,\rho_2)$ are rooted graphs, we write $(G_1,\rho_1) \cong_{\rho}
   (G_2,\rho_2)$ to denote the existence of a rooted isomorphism.

\subsection{Unimodular random graphs and distributional limits}
\label{sec:unimodular}

We begin with a discussion of unimodular random graphs and distributional limits.
One may consult the extensive reference of Aldous and Lyons \cite{aldous-lyons}.
The paper \cite{BS01} 
offers a concise introduction to distributional limits of finite planar graphs.
We briefly review some relevant points.

Let $\graphs$ denote the set of isomorphism classes of connected, locally finite graphs;
let $\rgraphs$ denote the set of {\em rooted} isomorphism classes of {\em rooted}, connected,
locally finite graphs.
Define a metric on $\rgraphs$ as follows:  $\dloc\left((G_1,\rho_1), (G_2,\rho_2)\right) = 1/(1+\alpha)$,
where
\[
   \alpha = \sup \left\{ r > 0 : B_{G_1}(\rho_1, r) \cong_{\rho} B_{G_2}(\rho_2, r) \right\}\,.
\]
$(\rgraphs, \dloc)$ is a separable, complete metric space. For probability measures $\{\mu_n\}, \mu$ on
$\rgraphs$, 
write $\{\mu_n\} \Rightarrow \mu$ when $\mu_n$ converges weakly to $\mu$ with respect to $\dloc$.
If $\{(G_n,\rho_n)\} \todl (G,\rho)$,
we say that $(G,\rho)$ is the {\em distributional limit} of the sequence $\{(G_n,\rho_n)\}$.

\paragraph{The Mass-Transport Principle}
Let $\rrgraphs$ denote the set of doubly-rooted isomorphism classes of doubly-rooted, connected, locally finite graphs.
A probability measure $\mu$ on $\rgraphs$ is {\em unimodular} if it obeys the following
{\em Mass-Transport Principle:}  For all Borel-measurable $F : \rrgraphs \to [0,\infty]$,
\begin{equation}\label{eq:mtp}
   \int \sum_{x \in V(G)} F(G,\rho,x) \,d\mu((G,\rho)) = \int \sum_{x \in V(G)} F(G,x,\rho)\,d\mu((G,\rho))\,.
\end{equation}
If $(G,\rho)$ is a random rooted graph with law $\mu$, and $\mu$ is unimodular,
we say that $(G,\rho)$ is a {\em unimodular random graph}.

\paragraph{Distributional limits of finite graphs}
As observed by Benjamini and Schramm \cite{BS01}, unimodular random graphs can be obtained from
limits of finite graphs.  Consider a sequence $\{G_n\} \subseteq \graphs$ of finite graphs,
and let $\rho_n$ denote a uniformly random element of $V(G_n)$.  Then $\{(G_n,\rho_n)\}$
is a sequence of $\rgraphs$-valued random variables,
and one has the following.

\begin{lemma}\label{lem:dl-unimodular}
   If $\{(G_n,\rho_n)\} \todl (G,\rho)$, then $(G,\rho)$ is unimodular.
\end{lemma}

   When $\{G_n\}$ is a sequence of finite graphs, we write $\{G_n\} \todl (G,\rho)$ for $\{(G_n,\rho_n)\} \todl (G,\rho)$
   where $\rho_n \in V(G_n)$ is chosen uniformly at random.

\paragraph{Unimodular random conformal graphs}

A {\em conformal graph} is a pair $(G,\omega)$, where $G$ is a connected, locally finite graph and
$\omega : V(G) \to \R_+$.
Let $\cgraphs$ and $\crgraphs$ denote the collections of
isomorphism classes of conformal graphs and
conformal rooted graphs, respectively.
As in \pref{sec:unimodular}, one can define a metric on $\crgraphs$ as follows:
$\dloc^*\left((G_1,\omega_1, \rho_1), (G_2, \omega_2, \rho_2)\right) = 1/(\alpha+1)$,
where
\[
   \alpha = \sup \left\{ r > 0 : B_{G_1}(\rho_1,r) \cong_{\rho} B_{G_2}(\rho_2, r) \textrm{ and }
   \vvmathbb{d}\!\left(\omega_1|_{B_{G_1}(\rho_1,r)},\omega_2|_{B_{G_2}(\rho_2,r)}\right) \leq \frac{1}{r}\right\}\,,
\]
where for two weights $\omega_1 : V(H_1) \to \R_+$ and $\omega_2 : V(H_2) \to \R_+$
on rooted-isomorphic graphs $(H_1,\rho_1)$ and $(H_2,\rho_2)$, we write
\begin{equation}\label{eq:omega-top}
   \vvmathbb{d}\!\left(\omega_1, \omega_2\right) \seteq \inf_{\psi : V(H_1) \to V(H_2)} \left\|\omega_2 \circ \psi - \omega_1\right\|_{\ell^{\infty}}\,,
\end{equation}
and the infimum is over all graph isomorphisms from $H_1$ to $H_2$ satisfying $\psi(\rho_1)=\rho_2$.

If $\{\mu_n\}$ and $\mu$ are probability measures on $\crgraphs$, we abuse
notation and write $\{\mu_n\} \todl \mu$ to denote weak convergence with respect to $\dloc^*$.
One defines unimodularity of a random rooted conformal graph $(G,\omega,\rho)$ analogously to \eqref{eq:mtp}:
It should now hold that for all Borel-measurable $F : \crrgraphs \to [0,\infty]$,
\[
   \int \sum_{x \in V(G)} F(G,\omega,\rho,x) \,d\mu((G,\omega,\rho))= \int \sum_{x \in V(G)} F(G,\omega,x,\rho)\,d\mu((G,\omega,\rho))\,.
\]
Indeed, such decorated graphs are a special case of the marked networks
considered in \cite{aldous-lyons}, and again it holds that
every distributional limit of finite unimodular random conformal graphs is a unimodular random conformal graph.

Suppose that $(G,\rho)$ is a unimodular random graph.  A {\em conformal weight on $(G,\rho)$}
is a unimodular random conformal graph $(G',\omega,\rho')$ such that $(G,\rho)$ and $(G',\rho')$
have the same law.  We will speak simply of a ``conformal metric $\omega$ on $(G,\rho)$.''
Only such unimodular metrics are considered in this work.

\subsubsection{Conformal growth rates under distributional limits}

In order to establish our main result, we need to pass from
a sequence of conformal metrics on finite graphs to a conformal metric
on the distributional limit.

\begin{theorem}\label{thm:QCG-limits}
   Consider $d,q \geq 1$.
   Suppose $\{(G_n,\rho_n)\}$ is a sequence of unimodular random graphs and $\{(G_n,\rho_n)\} \todl (G,\rho)$.
   If there is a function $h : \R_+ \to \R_+$ such that $h(R) \leq R^{o(1)}$ as $R \to \infty$,
   and a sequence of $L^q$-normalized unimodular random conformal graphs $\{(G_n,\omega_n,\rho_n)\}$
   satisfying
   \begin{equation}\label{eq:qcg-limits-growth}
      \|B_{\omega_n}(\rho_n,R)\|_{L^{\infty}} \leq R^d h(R)\,,
   \end{equation}
   then $\dimconfover^q(G,\rho) \leq d$.
   If the unimodular random graphs $\{(G_n,\rho_n)\}$ have uniform gauged $d$-dimensional growth, then
      $(G,\rho)$ has gauged $d$-dimensional growth.
\end{theorem}

The preceding theorem follows immediately from the next lemma.

\begin{lemma}\label{lem:weak2}
   Consider a sequence $\{(G_n, \omega_n, \rho_n)\}$ of unimodular random conformal graphs satisfying
   the following conditions:
   \begin{enumerate}
      \item $\{(G_n,\rho_n)\}$ has a distributional limit.
      \item $\limsup_{n \to \infty} \E[\omega_n(\rho_n)] < \infty$.
   \end{enumerate}
   Then $\{(G_n,\omega_n,\rho_n)\}$ has a subsequential weak limit in the metric $\dloc^*$.
\end{lemma}

\begin{proof}
   By passing to a subsequence and scaling, we may assume that
   \begin{equation}\label{eq:expec1}
      \E[\omega_n(\rho_n)] \leq 1 \qquad \forall n \geq 1\,.
   \end{equation}
   Let $\mu_n$ denote the law of $(G_n,\omega_n,\rho_n)$.  We will prove
   that the sequence $\{\mu_n\}$ is tight.  Since $(\crgraphs, \dloc^*)$
   is a complete, separable metric space, Prokhorov's theorem
   then implies that the sequence $\{\mu_n\}$ has a weak subsequential limit.

   To establish tightness, it suffices to exhibit a sequence $\{K_j \subseteq \crgraphs : j \geq 1\}$
   such that each $K_j$ is compact in the topology induced by $\dloc^*$ and
   \begin{equation}\label{eq:proh2}
      \lim_{j \to \infty} \lim_{n \to \infty} \mu_n(K_j) = 1\,.
   \end{equation}

   Let $\hat{\mu}_n$ denote the law of $(G_n,\rho_n)$.
   Since $(G_n,\rho_n)$ has a distributional limit and $(\rgraphs, \dloc)$ is complete,
   Prokhorov's theorem yields a sequence of compact sets $\{ \hat{K}_j \subseteq \rgraphs : j \geq 1\}$ such that
   \begin{equation}\label{eq:proh1}
      \lim_{j\to\infty} \lim_{n\to \infty} \hat{\mu}_n(\hat{K}_j) = 1\,.
   \end{equation}
   Denote the numbers:
   \[
      b_{j,k} \seteq \sup \left\{ |B_G(\rho, k)| : (G,\rho) \in \hat{K}_j \right\}\,.
   \]
   Since each $\hat{K}_j$ is compact, we have $b_{j,k} < \infty$ for all $j,k \geq 1$.

   Define the compact sets
   \[
      K_j \seteq \left\{ (G,\omega,\rho) : (G,\rho) \in \hat{K}_j \textrm{ and } \max_{x \in B_G(\rho, k)} \omega(x) \leq j k^2 b_{j,2k} \  \forall k \geq 1 \right\}\,.
   \]
   We are left to verify that \eqref{eq:proh2} holds.

   To that end, we apply the Mass-Transport Principle to $(G_n,\omega_n,\rho_n)$ with the flow
   \[
      F_{j,k}(G,\omega,x,y) \seteq \omega(y) \1_{\{\dist_G(x,y) \leq k\}} \1_{\{(G,x) \in \hat{K}_j \}}\,,
   \]
   yielding
   \begin{align*}
      j k^2 b_{j,2k} \Pr\left[(G_n,\rho_n) \in \hat{K}_j \textrm{ and } \max_{x \in B_{G_n}(\rho_n,k)} \omega_n(x) > j k^2 b_{j,2k}\right]
      &\leq
      \E\left[\1_{\{(G_n,\rho_n) \in \hat{K}_j\}} \sum_{y \in B_{G_n}(\rho_n,k)} \omega_n(y)\right] \\
      &=
      \E\left[\sum_{y \in V(G_n)} F_{j,k}(G_n,\omega_n,\rho_n,y)\right] \\
      &=
      \E\left[\sum_{x \in V(G_n)} F_{j,k}(G_n,\omega_n,x,\rho_n)\right]  \\
      &=
      \E\left[\omega_n(\rho_n) \sum_{x \in B_{G_n}(\rho_n,k)} \1_{\{(G,x) \in \hat{K}_j\}}\right] \\
      &\leq
      \E[\omega_n(\rho_n)] b_{j,2k}\,.
   \end{align*}
   Using \eqref{eq:expec1}, this gives
   \[
      \Pr\left[(G_n,\omega_n,\rho_n) \in K_j\right] \geq \Pr\left[(G_n,\rho_n) \in \hat{K}_j\right] - \frac{1}{j} \sum_{k \geq 1} \frac{1}{k^2}\,.
   \]
   In conjunction with \eqref{eq:proh1}, this yields \eqref{eq:proh2}.
\end{proof}

\subsection{Ahlfors regularity and systems of dyadic cubes}
\label{sec:metric-spaces}

Consider a complete, separable metric space $(X,d)$.
For $x \in X$ and two subsets $S,T \subseteq X$, we use the notations
$d(S,T) \seteq \inf_{x \in S, y \in T} d(x,y)$ and $d(x,S) = d(\{x\},S)$.
Define $\diam(S,d) \seteq \sup_{x,y \in S} d(x,y)$ and for $R \geq 0$, define
the closed balls
\[
   B_{(X,d)}(x, R) \seteq \{ y \in X : d(x,y) \leq R\}\,.
\]
We omit the subscript $(X,d)$ if the underlying metric space is clear from context.
We say that $(X,d)$ is {\em doubling} if there is a constant $\cD$ such that
every closed ball in $X$ can be covered by $\cD$ closed balls of half the radius,
and we let $\cD_{(X,d)}$ denote the infimal $\cD$ for which this holds.
$(X,d)$ is a {\em length space} if, for every $x,y \in X$, the distance
$d(x,y)$ is equal to the infimum of the length of continuous curves connecting
$x$ to $y$ in $X$.

If $\mu$ is a measure on the Borel $\sigma$-algebra of $X$, we refer to
$(X,d,\mu)$ as a {\em metric measure space.}
Such a space is said to be {\em Ahlfors $\beta$-regular} if there are constants $c_1,c_2 > 0$
such that
\[
   c_1 R^{\beta} \leq \mu\left(B(x,R)\right) \leq c_2 R^{\beta} \qquad \forall x \in X, R \in [0,\diam(X)]\,.
\]
It will occasionally be convenient to record the constants $c_1,c_2$, in which case we say that
$(X,d,\mu)$ is {\em $(c_1,c_2,\beta)$-regular.}
We recall the following elementary fact:
\begin{fact}\label{fact:ahlfors-doubling}
   If $(X,d,\mu)$ is Ahlfors $\beta$-regular for some $\beta > 0$, then $(X,d)$ is doubling,
   and $\cD_{(X,d)} \leq C$ for some constant $C=C(c_1,c_2,\beta)$ depending only on $c_1,c_2,\beta$.
\end{fact}

We will employ
an appropriate system of hierarchical dyadic partitions of a doubling metric space $(X,d)$
along the lines of \cite{Christ90} and \cite{David91}.
Deterministic and stochastic constructions of this type
are a basic tool in harmonic analysis and metric geometry
(see, e.g., \cite{LN05} and \cite{HK12}).

For our purposes, it will be easiest to use a construction from \cite{HK12}
which we summarize here.
Consider a metric space $(X,d)$.
A bi-infinite sequence $\bm{P} = \{\bm{P}_n : n \in \Z\}$ of partitions of $X$ is said to be a {\em hierarchical system}
if $\bm{P}_n$ is a refinement of $\bm{P}_{n+1}$ for all $n \in \Z$.
We say that $\bm{P}$ is {\em $\Delta$-adic} if 
\[
   S \in \bm{P}_n \implies \diam_{(X,d)}(S) \leq \Delta^n \qquad \forall n \in \Z\,.
\]

\begin{theorem}[\cite{HK12}]
   \label{thm:ensemble}
   Suppose $(X,d)$ is a doubling metric space.
   Then there are numbers $Q, \ell, \Delta \geq 2$ that depend only on $\cD(X,d)$ such that the following holds.
   There is a collection $\{ \bm{P}^{(1)}, \ldots, \bm{P}^{(Q)} \}$ of $\Delta$-adic hierarchical systems
   such that for every subset $S \subseteq X$ with $\diam_{(X,d)}(S) \leq \Delta^m$,
   there is a set
   \[
      C \in \bigcup_{i=1}^Q \bm{P}^{(i)}_{m+\ell}
   \]
   such that $S \subseteq C$.
\end{theorem}

\section{Quasi-packings and quasisymmetric invariance}
\label{sec:coarse-quasi}

We first demonstrate that the quasi-multiplicity condition \eqref{eq:bdd-mult}
can be replaced by a simpler assumption
whenever $(X,\dist)$ is an Ahlfors-regular length space
and one uses only strict balls instead of quasi-balls.

\subsection{Round balls, length spaces, and coarse packings}
\label{sec:round-bodies}

Let $\cB$ denote the set of closed balls in $(X,\dist)$.
Say that a graph $G$ is {\em $(\tau,M)$-coarsely packed in $(X,\dist)$}
if there is a map $\Phi : V(G) \to \cB$ satisfying \eqref{eq:coarse-tan} as well as
   \begin{equation}\label{eq:weak-mult}
      \# \{ v \in V(G) : x \in \Phi(v) \} \leq M\qquad \forall x \in X\,.
   \end{equation}
Recall that $G$ {\em coarsely packs in $(X,\dist)$} if it is $(\tau,M)$-coarsely packed for some $\tau,M \geq 1$.
Our goal in this section is to provide conditions on $(X,\dist)$ under
which coarse packings yield quasi-packings.

\begin{theorem}\label{thm:weak-to-coarse}
   Suppose that $(X,\dist,\mu)$ is an Ahlfors $d$-regular metric measure space
	and additionally that $(X,\dist)$ is a length space.
   Then for every locally finite graph $G$:
   \[
      \textrm{$G$ coarsely packs in $(X,\dist)$} \implies \textrm{$G$ quasi-packs in $(X,\dist)$}\,.
   \]
   Quantitatively, if $G$ is $(\tau,M)$-coarsely packed in $(X,\dist)$, then it is
   $(\tau',M)$-quasi-packed in $(X,\dist)$ with $\tau' \leq C \tau$, for some
   $C=C(X,\dist)$.
\end{theorem}

We will prove the theorem after establishing a few preliminary results.
Assume now that $\cX=(X,\dist,\mu)$ is a complete, separable metric measure space.
A Borel set $S \subseteq X$ is said to be {\em $\eta$-round}
if the following holds:  For every ball $B$ in $X$ whose center lies in $\bar{S}$
(the closure of $S$)
and for which $S \nsubseteq B$, we have
\begin{equation}\label{eq:eta-round}
   \mu(S \cap B) \geq \eta \cdot \mu(B)\,.
\end{equation}
Say that $\cX$ is {\em $\eta$-round} if every ball in $\cX$ is $\eta$-round,
and that $\cX$ is {\em uniformly round} if it is $\eta$-round for some $\eta > 0$.
For instance, $\R^d$ with the Euclidean metric is $2^{-d}$-round.

We recall that the measure $\mu$ is said to be {\em doubling}
if there is a constant $C \geq 1$ such that
\begin{equation}\label{eq:mudouble}
   \mu(B(x,2r)) \leq C \mu(B(x,r))
\end{equation}
for all $x \in X$ and $r \geq 0$.

\begin{lemma}
   If $\cX$ is a length space and $\mu$ is doubling, then
   $\cX$ is uniformly round.
   In particular, if \eqref{eq:mudouble} holds for some $C \geq 1$, then $\cX$
   is $1/(2C)$-round.
\end{lemma}

\begin{proof}
   Let $B_0=B(x,r)$.
   Consider any $y \in B_0$ and $r' < r$.
   Since $(X,\dist)$ is a length space, there is a point $z \in B_0$
   with $\dist(y,z)+\dist(z,x)=\dist(x,y)$ and satisfying
\begin{align*}
   \dist(x,z) &\leq r-r'\,, \\
   \dist(y,z) &\leq r'\,.
\end{align*}
In particular, it holds that $B(z,r') \subseteq B(y,r') \cap B(x,r)$,
implying that
\[
   \mu\left(B_0 \cap B(y,r')\right) \geq \mu(B(z,r')) \geq C^{-1} \mu(B(z,2r')) \geq C^{-1} \mu(B(y,r'))\,.\qedhere
\]
\end{proof}

   We will require the following elementary fact which states that a point in an
   Ahlfors $d$-regular space cannot be near too many pairwise-disjoint
   $\eta$-round bodies of large diameter.

\begin{lemma}\label{lem:touching}
	Suppose $\cX$ is $(c_1,c_2,d)$-regular and
   $S_1, S_2, \ldots, S_K \subseteq X$ are $\eta$-round sets
that satisfy
\begin{equation}\label{eq:mult-touching}
	\# \{ i \in \{1,\ldots,K\} : y \in S_i \} \leq s\qquad \forall y \in X\,,
\end{equation}
	and furthermore there is some $x \in X$ such that
   \[
      \max_{i \in [K]} \dist(x,S_i) < \alpha \cdot \min_{i \in [K]} \diam(S_i)\,,
   \]
   Then,
   \[
      K \leq s \frac{c_2}{c_1 \eta} (1+2\alpha)^d\,.
	\]
\end{lemma}

\begin{proof}
   Let $\lambda = \max_{i \in [K]} \dist(x,S_i)$, and let $\{x_i\}$ be a collection
   of points such that $x_i \in \bar{S}_i$ and $\dist(x,x_i) \leq \lambda$.
   Consider the balls $B_i = B(x_i, \lambda/(2 \alpha))$.  By assumption,
   $\diam(S_i) > \lambda/\alpha$, hence $S_i \nsubseteq B_i$.
   Thus by the definition of $\eta$-round,
   \[
      \mu(S_i \cap B_i) \geq \eta \mu(B_i) \geq \eta c_1 (\lambda/(2\alpha))^d\,,
   \]
   where the latter inequality follows from the Ahlfors regularity of $\cX$.
   But the sets $\{S_i\}$ satisfy $S_i \cap B_i \subseteq B(x, \lambda(1+1/(2\alpha))$ for every $i \in [K]$ and \eqref{eq:mult-touching},
	implying that
   \[
      K \eta c_1 (\lambda/(2\alpha))^d \leq s\cdot \mu\left(\vphantom{\bigoplus\vphantom{\bigoplus}}B(x,\lambda(1+1/(2\alpha))\right)
\leq s c_2 \lambda^d (1+1/(2\alpha))^{d}\,,
   \]
   where again the final inequality uses the Ahlfors $d$-regularity.
\end{proof}

\begin{proof}[Proof of \pref{thm:weak-to-coarse}]
	Consider $\Phi : V(G) \to \cB$
   and suppose that \eqref{eq:weak-mult} holds for some $M'$.
   Let $v_1, \ldots, v_{M} \in V(G)$ be such that
	$B(x,r) \cap \Phi(v_i) \neq \emptyset$ and $\diam(\Phi(v_i)) \geq r$
   for each $i=1,\ldots,M$.
   Under our assumptions, for some
	$c_1,c_2,\eta > 0$, \pref{lem:touching} (applied with $s=M'$ and $\alpha=1$) yields
\[
	M \leq M'\frac{c_2}{c_1 \eta} 3^d\,\qedhere
\]
\end{proof}

\subsection{Quasisymmetric stability}
\label{sec:stability}

Recall that if $(X,d_X)$ and $(Y,d_Y)$ are metric spaces, then
a map $f : X \to Y$ is {\em $\eta$-quasisymmetric} if
there is a homeomorphism $\eta : [0,\infty) \to [0,\infty)$ such that
      for all $x,y,z \in X$:
\begin{equation}\label{eq:qs}
   \frac{d_Y(f(x),f(y))}{d_Y(f(x),f(z))} \leq \eta\left(\frac{d_X(x,y)}{d_X(x,z)}\right)\,.
\end{equation}
The spaces $(X,d_X)$ and $(Y,d_Y)$ are said to be {\em quasisymmetrically equivalent}
if, for some $\eta$, there is an $\eta$-quasisymmetric bijection from $X$ to $Y$.

A metric space $(X,d_X)$ is {\em uniformly perfect}
if there is a number $\lambda \geq 1$ so that for every $x \in X$ and $r > 0$,
the set $B_X(x,r) \setminus B_X(x,r/\lambda)$ is non-empty whenever $X \setminus B_X(x,r)$
is non-empty.
We refer to \cite[\S 11]{Heinonen01} for background on these notions and their interplay.
In particular, one has the following basic facts.

\begin{lemma}\label{lem:perfect-qs}
   If $(X,d_X)$ and $(Y,d_Y)$ are quasisymmetrically equivalent,
   then $(X,d_X)$ is uniformly perfect if and only if $(Y,d_Y)$ is uniformly perfect.
\end{lemma}

\begin{lemma}\label{lem:diams}
   If $f : X \to Y$ is $\eta$-quasisymmetric and $A \subseteq B \subseteq X$, then
   \[
      \frac{1}{2 \eta\left(\frac{\diam_X(B)}{\diam_X(A)}\right)}\leq
      \frac{\diam_Y(f(A))}{\diam_Y(f(B))} \leq \eta\left(\frac{2\, \diam_X(A)}{\diam_X(B)}\right)\,.
   \]
\end{lemma}

\begin{lemma}[{\cite[Lem. 2.5]{Tyson01}}]
\label{lem:quasiballs}
If $f : X \to Y$ is $\eta$-quasisymmetric and $S$ is a $\tau$-quasi-ball in $X$, then
	$f(S)$ is a $2 \eta(\tau)$-quasi-ball in $Y$.
\end{lemma}

The main result of this section is that, for uniformly perfect spaces,
if $X$ and $Y$ are quasisymmetrically equivalent, then the
classes of graphs that quasi-pack into $X$ and $Y$ coincide.

\begin{theorem}\label{thm:qs-invariant}
   Suppose $(X,d_X)$ and $(Y,d_Y)$ are quasisymmetrically equivalent
   uniformly perfect spaces.  Then there is a constant $K \geq 1$
   such that a locally finite graph $G$ is
   $(\tau,M)$-quasi-packed in $(X,d_X)$ if and only if it is $(\tau',M)$-quasi-packed in $(Y,d_Y)$,
   and moreover $K^{-1} \tau \leq \tau' \leq K \tau$.
\end{theorem}

\begin{proof}
   Let $f : X \to Y$ be an $\eta$-quasisymmetric bijection.
   Since $f$ is a bijection, we will assume that $X=Y$.
   We use $B_X$ and $B_Y$ to denote balls in the metrics $d_X$ and $d_Y$, respectively.
   Assume that $(X,d_X)$ is uniformly perfect with constant $\lambda \geq 1$.

   Suppose that $G$ is $(\tau,M)$-quasi-packed in $(X,d_X)$, and let
   $\Phi : V(G) \to \cB_{\tau}$ denote a mapping that verifies \eqref{eq:coarse-tan}
   and \eqref{eq:bdd-mult}.
	Our goal is to establish that $f \circ \Phi$ witnesses a $(\tau',M)$-quasi-packing in $(Y,d_Y)$
	for some $\tau' \geq 1$.
	For every $v \in V(G)$, let $(x_v, r_v)$ be such that
	$B_X(x_v, r_v) \subseteq \Phi(v) \subseteq B_X(x_v, \tau r_v)$.

\medskip
\noindent
{\bf Quasi-tangency.}
	Consider $\{u,v\} \in E(G)$ and suppose that $\diam_Y(\Phi(u)) \geq \diam_Y(\Phi(v))$.
	Observe that
	\eqref{eq:coarse-tan} implies there is a $z \in \Phi(u) \cap B_X(x_v, 2 \tau^2 r_v)$.
   Thus \pref{lem:diams} gives
	\begin{align*}
		d_Y(\Phi(u),\Phi(v)) &\leq \diam_Y(B_X(x_v, 2\tau^2 r_v)) \\ 
&\leq
\diam_Y(\Phi(v)) \cdot 2\eta\left(\frac{4\tau^2 r_v}{\diam_X(\Phi(v))}\right)\\
&\leq
\diam_Y(\Phi(v)) \cdot 2\eta\left(4\lambda\tau^2\right)\,,
	\end{align*}
where the second inequality employs \pref{lem:diams},
and in the last inequality we have used that $X$ is uniformly perfect.
Employing \pref{lem:quasiballs},
we have thus verified that \eqref{eq:coarse-tan} holds for $f \circ \Phi$ with
$\tau'=2 \eta(4\lambda\tau^2)$.

\medskip
\noindent
{\bf Quasi-multiplicity.}
Consider now some $x' \in Y$, $r' > 0$, and a subset $S \subseteq V(G)$
such that $B_Y(x',r') \cap \Phi(v) \neq \emptyset$ and
$\diam_Y(\Phi(v)) \geq \tau' r'$ for all $v \in S$.

Let $D \seteq \diam_X(B_Y(x',r'))$.
Fix $v \in S$ and $z \in B_Y(x',r') \cap \Phi(v)$.
Choose $z' \in \Phi(v)$ such that $d_Y(z,z') \geq \tau' r'/2$.
Choose $z'' \in B_Y(x',r')$ so that $d_X(z,z'') \geq D/2$.
Note that $f^{-1}$ is $\eta'$-quasisymmetric with $\eta'(t)=1/\eta^{-1}(1/t)$, therefore
from \eqref{eq:qs}:
\begin{equation}\label{eq:almostthere}
   \frac{\diam_X(B_Y(x',r'))}{\diam_X(\Phi(v))} \leq
   \frac{\diam_X(B_Y(x',r'))}{d_X(z,z')} \leq
   2 \frac{d_X(z,z'')}{d_X(z,z')} \leq \eta'\left(\frac{d_Y(z,z'')}{d_Y(z,z')}\right) \leq \eta'\left(\frac{4}{\tau'}\right)\,.
\end{equation}

Choose $\tau'$ large enough so that $\eta'(4/\tau') \leq 1/\tau$.
Let $r \seteq \diam_X(B_Y(x',r'))$ and fix any $x \in B_Y(x',r')$.
By construction, $B_X(x,r) \cap \Phi(v) \neq \emptyset$ for every $v \in S$.
By \eqref{eq:almostthere} and our choice of $\tau'$, we have
\[
   \diam_X(\Phi(v)) \geq \tau r  \quad \forall v \in S\,.
\]
Applying the quasi-multiplicity condition \eqref{eq:bdd-mult} to $\Phi$,
we see that $|S| \leq M$.
We have thus verified that \eqref{eq:bdd-mult} holds also for $f \circ \Phi$
with $\tau'$ is chosen appropriately.
\end{proof}

\section{Discrete conformal metrics on $d$-dimensional graphs}
\label{sec:sphere-packings}

We first state the main technical result of this section.
Recall the definition \[d_* \seteq \max(d,2)\,.\]

\begin{theorem}\label{thm:finite-sphere-packing}
   For every $d,\tau,M \geq 1$ and $c_1,c_2 > 0$,
   there is a number $C \geq 1$ such that the following holds.
   Suppose $G=(V,E)$ is a finite graph that is $(\tau,M)$-quasi-packed in a
   $(c_1,c_2,d)$-regular space $\cX$.
   Then for every $R \geq 0$, there is a conformal weight $\omega : V \to \R_+$
   that satisfies
   \begin{equation}\label{eq:norm-weight}
      \frac{1}{|V|} \sum_{x \in V} \omega(x)^{d_*} = 1\,,
   \end{equation}
   and 
   such that 
   \begin{equation}\label{eq:desired25}
      \max_{x \in V} |B_{\omega}(x,R)| \leq C R^{d_*}\,.
   \end{equation}
\end{theorem}
Combining this with \pref{thm:QCG-limits} yields \pref{thm:intro-scaled}.

\subsection{Properties of quasi-packings}
\label{sec:coarse-packings}

Suppose that $G$ is $(\tau,M)$-quasi-packed in a $(c_1,c_2,d)$-regular space $(X,\dist,\mu)$
for some $d,\tau,M \geq 1$ and $c_1,c_2 > 0$.
Let $\{ S_v : v \in V(G)\}$ denote a family of $\tau$-quasi-balls in $X$ that
satisfy \eqref{eq:coarse-tan} and \eqref{eq:bdd-mult}.
We now collect all the properties we will require of such a ``packing''
in proving the main theorem.

Throughout this section
and the next, we will use the asymptotic notation $A \lesssim B$
to denote that $A \leq C\cdot B$ for some constant $C$ that depends only the
parameters $d,c_1,c_2,\tau,M$.
We use $A \asymp B$ to denote the conjunction of $A \lesssim B$ and $B \lesssim A$.
\begin{enumerate}
   \item
For every $v \in V(G)$,
\begin{equation}\label{eq:regular}
   \diam(S_v)^d \asymp \mu(S_v)\,.
\end{equation}
This follows immediately from the definition of $(c_1,c_2,d)$-regular.
\item
For every $x \in X$,
\[
	\# \{ v \in V(G) : x \in S_v \} \lesssim 1\,.
\]
This follows from \eqref{eq:bdd-mult} with $r=0$.
\item
   For every $\{u,v\} \in E(G)$ and $x \in S_u, y \in S_v$:
   \begin{equation}\label{eq:set-dist}
      \dist(x,y) \lesssim \diam(S_u) + \diam(S_v)\,.
   \end{equation}
   This follows immediately from \eqref{eq:coarse-tan}.
\item
   Consider a Borel set $Y \subseteq X$.
   It holds that
   \begin{equation}\label{eq:vmult}
      \sum_{v \in V(G) : S_v \subseteq Y} \mu(S_v) \lesssim \mu(Y) \lesssim \diam(Y)^d\,.
   \end{equation}
The first inequality follows from (2) and the second from Ahlfors regularity.
\item For any $\lambda \geq 1$, there is a number $C=C(\lambda,c_1,c_2,d,\tau)$
such that for all $x \in X$ and $r > 0$,
   \begin{equation}\label{eq:touching}
      \# \left\{v \in V(G) : \diam(S_v) \geq r \textrm{ and } \dist(x,S_v) \leq \lambda r \right\}
\lesssim C\,.
   \end{equation}
We derive this from \eqref{eq:bdd-mult} as follows.
Cover $B(x, \lambda r)$ by balls $B_1, B_2, \ldots, B_{C'}$ of radius $r/\tau$, where $C'=C'(c_1,c_2,d,\tau,\lambda)$.
Now apply \eqref{eq:bdd-mult} to each $B_i$ separately to obtain \eqref{eq:touching} with $C \leq C' M$.
\end{enumerate}

\subsection{Discrete uniformization}
\label{sec:conformal-weight}

Our proof of \pref{thm:finite-sphere-packing} is inspired by
the ``isolation lemma'' of Benjamini and Schramm \cite{BS01} (see also \cite{BC11,Gill14}).
Suppose $G=(V,E)$ is sphere-packed in $\R^d$.  
When the spheres $\{S_v : v \in V\}$ in the packing
have comparable radii, the background Euclidean metric provides
a reasonable conformal weight; one sets $\omega(v)$
proportional to the radius of the sphere $S_v$.

\begin{figure}
      \begin{center}
            \subfigure[An isolated accumulation point]{ \includegraphics[width=6cm]{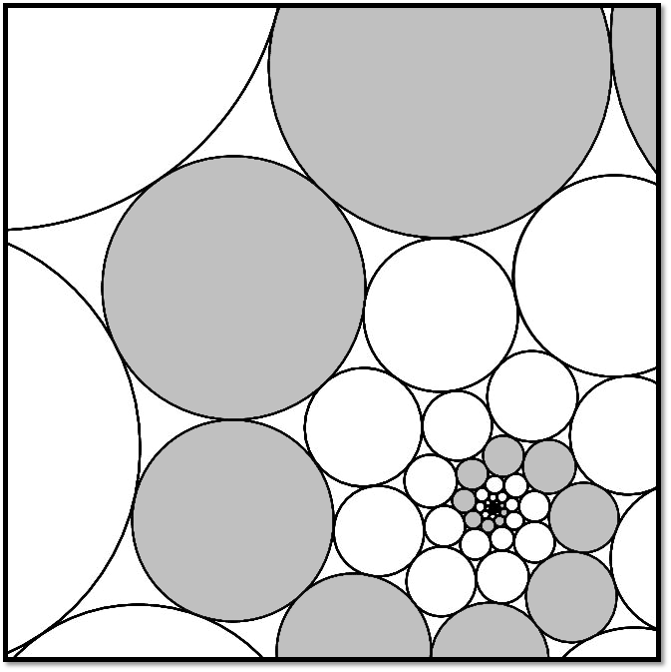}\label{fig:isolated} } \hspace{0.8in}
         \subfigure[A continuum of accumulation points]{  \includegraphics[width=6cm]{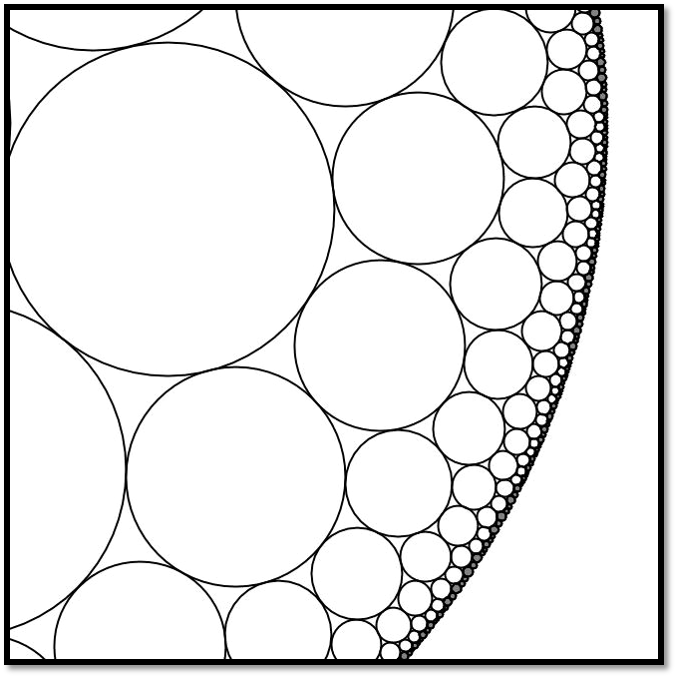}\label{fig:hyperbolic} }
      \caption{Accumulation points}
   \end{center}
      \end{figure}

Difficulties arise when the radii degenerate,
for instance near an accumulation point (in the case of infinite $G$); see, for example,
\pref{fig:isolated}.
But if one imagines an {\em isolated} accumulation point as a cone, then it becomes
rather tame: If we think of it as a metric on $\vvS^{d-1} \times [0,\infty)$, where the $d$th dimension is along
the axis of the cone, then we merely need to do a ``$1$-dimensional uniformization'' along the axis
(this can be seen in the use of the concavity of $x \mapsto x^{1/d}$ in \pref{cor:helper5} below).
It would be problematic if the accumulation points themselves accumulated, e.g.,
as for a circle packing of a triangulation of the hyperbolic plane (e.g., \pref{fig:hyperbolic}).
But the Benjamini-Schramm lemma asserts that this cannot happen for distributional limits of
finite graphs packed in $\R^d$.

   By default, we use the notation $\diam(\cdot)$ to denote the diameter
   in the metric $\dist$.  When we consider another metric,
   it will be explicitly specified.

   \subsubsection{Construction of the conformal weight}
   \label{sec:construction}

   Suppose now that $G=(V,E)$ is a finite graph that is $(\tau,M)$-quasi-packed in $(X,\dist,\mu)$.
	To each $v \in V$, associate a $\tau$-quasi-ball $S_v \subseteq X$ so that
	\pref{sec:coarse-packings}(1)--(5) are satisfied.

   Assume that $k \geq 3$ is given.  We will establish the existence of a metric $\omega : V \to \R_+$
   that satisfies $\frac{1}{|V|} \sum_{x \in V} \omega(x)^d \lesssim 1$ and such that any subset
   $U \subseteq V$ with $|U| = 2^k$ satisfies $\diam_{\omega}(U) \gtrsim 2^{k/d_*}$.
   This suffices to establish \pref{thm:finite-sphere-packing}.

   \medskip

	Identify $v$ with
   an arbitrary point in $S_v$ so that we may consider $V \subseteq X$.
   Define $\omega_0(v) \seteq \mu(S_v)^{1/d}$.
   Then \eqref{eq:regular} gives:
   \begin{equation}\label{eq:omega-compare}
      \diam(S_v) \asymp \omega_0(v)\,.
   \end{equation}

   Let $\bm{P} = \{ \bm{P}_n : n \in \Z \}$ denote a $\Delta$-adic hierarchical system in $X$
(recall \pref{sec:metric-spaces}).
Define \[\hat{\bm{P}} \seteq \left\{\vphantom{\bigoplus} (C,n) : n \in \Z, C \in \bm{P}_n \right\}.\]
   Consider a positive integer $s \lesssim 1$ to be chosen soon.

   \medskip
   \noindent
   {\bf The level of a cube.}
   For a pair $(C,n) \in \hat{\bm{P}}$, define
   \[
      \level_{\bm{P}}(C,n) \seteq \max \left\{ j \in \N :
      |(V \cap C) \setminus C')| \geq 2^j  \textrm{ for all } C' \in \bm{P}_{n-s}\right\}\,.
   \]
   The relevance of this definition is as follows.  If $\level_{\bm{P}}(C,n)=j$, then we are
   witnessing a ``feature'' of size $\approx 2^j$ that will not be fully seen by
   any cube at any lower scale.  (For technical reasons, we actually shift by $s$ scales, but
   $s \lesssim 1$.)

   Thus we need to ``uniformize'' this feature at the current scale.
   Since we are trying to ensure $d$-dimensional volume growth,
   it should not be that this set of $2^j$ points is contained
   in a set of $\dist_{\omega}$-diameter significantly less than $2^{j/d}$ (for $d \geq 2$).

   Let us first present a heuristic analysis.
   Suppose we consider a cube $C \in \bm{P}_n$ of diameter at most $\Delta^n$ and $\level_{\bm{P}}(C,n)=j$.
   Moreover, suppose that for $v \in V \cap C$, it holds that $\omega_0(v) \lesssim \Delta^n$. (This is
   the case of ``small bodies'' in the arguments below; large bodies are handled by a separate argument.)

   Then we should scale the metric $\omega_0$ by $\approx \Delta^{-n} 2^{j/d}$ to ensure
   that we inflate this set to large enough diameter.
   (This is assuming that $\diam(V \cap C) \approx \Delta^n$; if the bulk $V \cap C$
   has much smaller diameter, this feature will be detected at the correct scale
   in some other hierarchical system.)
   Thus we should endow the vertices $v \in V \cap C$ with weight $\omega(v) \geq \beta \omega_0(v)$,
   where $\beta \approx \Delta^{-n} 2^{j/d}$.

   Consider now how much conformal weight we have spent.
   By a simple volume argument \eqref{eq:vmult}, the total $\ell_d$-weight allocated is
   proportional to
   \[
      \Delta^{-nd} 2^{j} \sum_{v \in V \cap C} \omega_0(v)^d
      \lesssim \Delta^{-nd} 2^j \Delta^{nd} \lesssim 2^j\,.
   \]
   Thus if we hope to keep the total $\ell_d$-weight bounded,
   it should be that we cannot detect too many level-$j$ features.
   This is the content of the next lemma which follows \cite[Lem 2.3]{BS01}.

   \begin{lemma}\label{lem:BSmagic}
      For all integers $j \geq 0$,
   \begin{equation}\label{eq:BSmagic}
      \# \left\{ (C,n) \in \hat{\bm{P}} : \level_{\bm{P}}(C,n)=j \right\} \leq \frac{2s|V|}{2^j}\,.
   \end{equation}
   \end{lemma}

   \begin{proof}
      Fix $j \geq 0$.
      Denote \[ \equivclass{\sigma} \seteq \left\{ n \in \Z : n \equiv \sigma\ (\bmod\ s) \right\}\,.  \]
      We will prove that for $\sigma \in \{0,1,\ldots,s-1\}$,
      \begin{equation}\label{eq:BSmagic2}
         \# \left\{ (C,n) \in \hat{\bm{P}} : \level_{\bm{P}}(C,n)=j \textrm{ and } n \in \equivclass{\sigma}\right\} \leq \frac{2|V|}{2^j}\,.
      \end{equation}

      Fix $\sigma \in \{0,1,\ldots,s-1\}$.
      For a pair $(C,n) \in \hat{\bm{P}}$, define the set of children
      \[
         \children(C,n) \seteq \left\{ C' \subseteq C : C' \in \bm{P}_{n-s} \right\}\,.
      \]

      Define a ``flow'' $F : (2^X \times \equivclass{\sigma}) \times (2^X \times \equivclass{\sigma}) \to \R$
      ``up'' the hierarchical system $\bm{P}$ as follows: For every $n \in \equivclass{\sigma}$,
      \[
         F\!\left(\vphantom{\bigoplus}(C',n-s), (C,n)\right) = \begin{cases}
            \min \{ 2^{j-1}, |C' \cap V| \}   &         C \in \bm{P}_n, C' \in \children(C,n)  \\
            0 & \textrm{otherwise.}
         \end{cases}
      \]
      Define also:
      \newcommand{\fin}{\mathrm{in}}
      \newcommand{\fout}{\mathrm{out}}
      \begin{align*}
         F_{\fin}(C,n) &\seteq \sum_{(C',n') \in \hat{\bm{P}}} F\left((C',n'),(C,n)\right), \\
         F_{\fout}(C,n) &\seteq \sum_{(C',n') \in \hat{\bm{P}}} F\left((C,n),(C',n')\right), \\
         F_{\fin}^{(n)} &\seteq \sum_{C \in \bm{P}_n} F_{\mathrm{in}}(C,n)\,.
      \end{align*}
      We make three observations:
      \begin{enumerate}
         \item First, notice that flow only goes ``up'' from a child set to a parent set,
            and thus from a lower level to a higher level:
            \begin{equation*}
            F\left((C',n'),(C,n)\right) > 0 \implies n,n' \in [\sigma], n=n'+s, C' \in \children(C,n)\,.
         \end{equation*}
      \item The flow out of $(C,n)$ is always at most the flow into $(C,n)$: $F_{\fout}(C,n) \leq F_{\fin}(C,n)$.
            This is because for $C \in \bm{P}_n$,
            \[
               \sum_{C' \in \children(C,n)} |C' \cap V| = |C \cap V|\,.
            \]
         \item
            When $\level_{\bm{P}}(C,n)=j$, the flow leaving $(C,n)$ is less than the flow
            entering $(C,n)$ by a least $2^{j-1}$
            because by definition of $\level_{\bm{P}}(C,n)$, 
            \[
               \sum_{C' \in \children(C,n)} \min \{ 2^{j-1}, |C' \cap V| \} \geq 2^j\,.
            \]
            In particular, combining this with observation (2) yields, for every $n \in \Z$,
            \begin{equation}\label{eq:only-one}
               F_{\fin}^{(n+1)} \leq F_{\fin}^{(n)} - 2^{j-1} \# \{ C \in \bm{P}_n : \level_{\bm{P}}(C,n)=j \}\,.
            \end{equation}
      \end{enumerate}

      On the other hand, let $n_0 \in \equivclass{\sigma}$ be small enough so that every $C \in \bm{P}_{n_0}$ contains
      at most one point of $V$.  Then $F_{\fin}^{(n)} \leq |V|$ for all $n \leq n_0$.
      Combining this with \eqref{eq:only-one} and the fact that $F \geq 0$ implies \eqref{eq:BSmagic2}.
   \end{proof}

   Let us now assume additionally that $\bm{P}$ is $\Delta$-adic for some $2 \leq \Delta \lesssim 1$ to
   be fixed momentarily.
   Given $S \subseteq X$ and a parameter $n \in \Z$, we define the enlargements
   \[
      N(S,R) \seteq \left\{ x \in X : \dist(x,S) \leq R\right\}\,.
   \]
   Define also the truncated level function:
   \[
      \levelt_{\bm{P}}(C,n) \seteq \min \{ \level_{\bm{P}}(C,n), k \}\,,
   \]
   where we recall that $k$ is the parameter defined at the beginning of \pref{sec:construction}.

   \begin{remark}
   The motivation for this truncation lies in the definition \eqref{eq:theta-def} below,
   and the fact that we are only attempting to establish \eqref{eq:desired25} for a single
   value of $R$ or, equivalently, a single value of $k$.
   Considering ``features'' with level larger than $k$ would incur a quantitative overhead
   that doesn't allow us to obtain a constant $C$ in \eqref{eq:desired25}
   that is independent of $R$.
   \end{remark}

   Note that \pref{lem:BSmagic} gives
   \begin{equation}\label{eq:levelt}
      \# \left\{ (C,n) \in \hat{\bm{P}} : \levelt_{\bm{P}}(C,n) = j \right\} \leq \frac{4s|V|}{2^j}\,,
   \end{equation}
   where the extra factor of $2$ comes from the consequence
   \[
      \# \left\{ (C,n) \in \hat{\bm{P}} : \level_{\bm{P}}(C,n) \geq j \right\} \leq \frac{4s|V|}{2^j}\,.
   \]

   Recall that $d_* = \max(d,2)$.
   For every $(C,n) \in \hat{\bm{P}}$, we define a function $\theta_{\bm{P}}^{(C,n)} : V \to \R$ as follows:
   \begin{equation}\label{eq:theta-def}
      \theta^{(C,n)}_{{\bm{P}}}(v) \seteq \begin{cases} 
         {\displaystyle\frac{2^{\levelt_{\bm{P}}(C,n)/d_*}}{\left(1+k-\levelt_{{\bm{P}}}(C,n)\right)^{2/d_*}}}
         \cdot \min \left\{ \Delta^{-n}, \frac{1}{\omega_0(v)} \right\} & \textrm{if } S_v \cap N(C,2\tau \Delta^n) \neq \emptyset
         \,, \vspace*{0.1in}\\
                              0 & \textrm{otherwise.}
                           \end{cases}
   \end{equation}

   Define a conformal weight $\omega_{{\bm{P}}} : V \to \R_+$ by
   \[
      \omega_{{\bm{P}}}(v) \seteq \omega_0(v) \left(\sum_{(C,n) \in \hat{\bm{P}}} \left(\theta^{(C,n)}_{{\bm{P}}}(v)\right)^{d_*}\right)^{1/d_*}
   \]

   The $1/\omega_0(v)$ factor in \eqref{eq:theta-def} is there to handle the case
   of a set $S_v$ with $\diam(S_v) > \Delta^n$ intersecting the neighborhood of $C$.
   Denote 
   \begin{equation}\label{eq:ENC}
      E_n(C) \seteq \left\{ \vphantom{\bigoplus}
               v \in V : \omega_0(v) > \Delta^n \textrm{ and } S_v \cap N(C,2\tau \Delta^n)\neq \emptyset \right\},
   \end{equation}
   From \eqref{eq:omega-compare}, we have $v \in E_n(C) \implies \diam(S_v) \gtrsim \omega_0(v) \geq \Delta^n$, and therefore \eqref{eq:touching} implies that
   \begin{equation}\label{eq:exceptional}
      |E_n(C)| \lesssim 1\quad \textrm{for all} \quad (C,n) \in \hat{\bm{P}}\,.
   \end{equation}
   Now write:
\begin{equation}\label{eq:weight-bound}
      \sum_{v \in V} \omega_{{\bm{P}}}(v)^{d_*} 
      = \sum_{j=0}^{k}\frac{2^{j}}{(1+k-j)^2} \sum_{n \in \Z} \sum_{\substack{C \in {\bm{P}}_n :\\ \levelt_{\bm{P}}(C,n)=j}}
      \left(|E_n(C)|
                  +\Delta^{-{d_*} n} \sum_{\substack{v \in V :\\ S_v \cap N(C,2\tau\Delta^n) \neq \emptyset \\ \omega_0(v) \leq \Delta^{n}}} \omega_0(v)^{d_*}
      \right)\,.
\end{equation}

   From \eqref{eq:omega-compare}, we have $\diam(S_v) \leq K_0 \omega_0(v)$ for some $1 \leq K_0 \lesssim 1$ and every $v \in V$.
   Thus in the case $d = d_*$, for a fixed $C \in {\bm{P}}_n$, we have
\begin{align*}
      \Delta^{-d_* n}\sum_{\substack{v \in V :\\ S_v \cap N(C,2\tau\Delta^n) \neq \emptyset \\ \omega_0(v) \leq \Delta^{n}}}  \omega_0(v)^{d_*}
      &=
      \Delta^{-d n}\sum_{\substack{v \in V :\\ S_v \cap N(C,2\tau\Delta^n) \neq \emptyset \\ \omega_0(v) \leq \Delta^{n}}}  \mu(S_v) \\
      &\stackrel{\mathclap{\eqref{eq:vmult}}}{\lesssim}
      \Delta^{-d n} \diam(N(C,(2\tau+K_0)\Delta^n))^d  \\
      &\lesssim 1\,.
\end{align*}

When $d < d_*$, use monotonicity of $\ell_p$ norms to write:
\[
      \left(\sum_{\substack{v \in V :\\ S_v \cap N(C,2\tau\Delta^n) \neq \emptyset \\ \omega_0(v) \leq \Delta^{n}}}  \left(\frac{\omega_0(v)}{\Delta^n}\right)^{d_*}\right)^{d/d_*}
      \leq
      \sum_{\substack{v \in V :\\ S_v \cap N(C,2\tau\Delta^n) \neq \emptyset \\ \omega_0(v) \leq \Delta^{n}}}  \left(\frac{\omega_0(v)}{\Delta^n}\right)^{d}
      = \Delta^{-dn} \sum_{\substack{v \in V :\\ S_v \cap N(C,2\tau\Delta^n) \neq \emptyset \\ \omega_0(v) \leq \Delta^{n}}} \mu(S_v)
      \stackrel{\eqref{eq:vmult}}{\lesssim}
      1\,.
\]

   Using this in \eqref{eq:weight-bound} together with \eqref{eq:exceptional},
   we conclude that
   \begin{align}
      \sum_{x \in V} \omega_{{\bm{P}}}(x)^{d_*} &\lesssim \nonumber
      \sum_{j=0}^{k} \frac{2^{j}}{(1+k-j)^2} \# \left\{ (C,n) \in \hat{\bm{P}} : \levelt_{\bm{P}}(C,n) = j\right\}
      \\& 
      \stackrel{\mathclap{\eqref{eq:levelt}}}{\leq}  |V| \sum_{j=0}^{k} \frac{4s}{(1+k-j)^2} \nonumber
      \\
      &\lesssim |V|\,.\label{eq:dvol}
   \end{align}

   Since $(X,\dist)$ is doubling, \pref{thm:ensemble} implies that for some positive integers
   $Q,\ell, \lesssim 1$ and $2 \leq \Delta \lesssim 1$, there
   is a sequence $\{\bm{P}^{(1)}, \ldots \bm{P}^{(Q)}\}$ of $\Delta$-adic hierarchical systems in $X$
   such that:
   \begin{equation}\label{eq:ensemble} 
      S \subseteq X,\  \diam(S) \leq \Delta^m \implies S \subseteq C \textrm{ for some } (C,m+\ell) \in \bigcup_{i=1}^Q \hat{\bm{P}}^{(i)}\,.
   \end{equation}
   Let us now set
   \begin{equation*}\label{eq:set-s}
      s \seteq \ell+4
   \end{equation*}
   in the preceding construction.
   To construct our final weight, we define
   \begin{equation}\label{eq:omega-construct}
      \omega \seteq \omega_{\bm{P}^{(1)}} + \cdots + \omega_{\bm{P}^{(Q)}}\,.
   \end{equation}
   It follows that
   \[
      \left(\frac{1}{|V|} \sum_{x \in V} \omega(x)^{d_*}\right)^{1/d_*} \lesssim \max \left\{
      \left(\frac{1}{|V|} \sum_{x \in V} \omega_{\bm{P}^{(i)}}(x)^{d_*}\right)^{1/d_*} : i =1,\ldots,Q \right\} \stackrel{\eqref{eq:dvol}}{\lesssim} 1\,,
   \]
   where in the first inequality we used the fact that $Q \lesssim 1$.

   \subsubsection{The growth bound}

   The next lemma finishes the proof of \pref{thm:finite-sphere-packing}.

   \begin{lemma}\label{lem:otherside}
   For every subset of vertices $U \subseteq V$ with $|U| = 2^k$,
   there is an index $i \in \{1,\ldots,Q\}$ satisfying
   \begin{equation}\label{eq:otherside}
      \diam_{\omega_{\bm{P}^{(i)}}}(U) \gtrsim 2^{k/d_*}\,.
   \end{equation}
   \end{lemma}

\begin{proof}
   Let us fix a subset $U \subseteq V$,
   and denote $D = \diam(U) > 0$.
   Let $n' \seteq \lceil \log_{\Delta} D\rceil + \ell$.
   Then by \eqref{eq:ensemble}, there is an index $i \in \{1,\ldots,Q\}$ such that
   $U \subseteq C$ for some $(C,n') \in \hat{\bm{P}}^{(i)}$.
   Let $\bm{P}=\bm{P}^{(i)}$.

   We now define inductively a sequence of 
   pairs $(C'_0,n'), (C'_1,n'-s), \ldots, (C'_{m'},n'-m's) \in \hat{\bm{P}}$
   as follows.
   \begin{itemize}
      \item Let $C'_0\seteq C$\,.
      \item If $|U \cap C'_i| \leq 1$, we set $m' \seteq i$ and stop.

         Otherwise, we choose $C'_{i+1} \in \bm{P}_{n-s(i+1)}$ to be an element of the set
         $\{C' \in \bm{P}_{n-s(i+1)}: C' \subseteq C'_i \}$ that maximizes $|U\cap C'|$.
   \end{itemize}

   Let us then pass to the maximal subsequence $\{(C_0,n_0), (C_1,n_1),\ldots,(C_m,n_m)\}$
   of the sequence
   $\{(C'_0, n), (C'_1,n-s), \ldots, (C'_{m'}, n-m's)\}$
   with $n_0 > n_1 > \cdots > n_m$
   and the property that
   \[
      n_i = \min \left\{ n  : \exists (C'_j,n'-js) \textrm{ with } n = n'-js \textrm{ and } C'_j \cap U = C_i \cap U \right\}\,.
   \]
   In other words, we enforce the property that 
   \begin{equation}\label{eq:N0} 
      C_{i+1} \cap U \neq C_i \cap U \quad\textrm{for each } i=0,1,\ldots,m-1\,.
   \end{equation}
   Define $C_{m+1}=\emptyset$.

   We have chosen the sequence $\{n_i\}$ in this way so that
   for every $i \in \{0,1,\ldots,m\}$,
   \begin{equation}\label{eq:thelevels}
   \levelt_{{\bm{P}}}(C_i,n_i) \geq \lfloor \log_2 |(U \cap C_i) \setminus C_{i+1}|\rfloor\,.
   \end{equation}
   From our choice of $s=\ell+4$
   and the fact that $\bm{P}$ is $\Delta$-adic with $\Delta \geq 2$,
   it holds that 
   \[
      \diam(C_{1}) \leq \Delta^{n'-s} \leq \Delta^{-3} D \leq \frac{D}{8}\,.
   \]
   Since $\diam(U)=D$,
   there must exist some $u_0 \in U$ such that
   \begin{equation}\label{eq:large-dist}
      \dist(u_0, C_{1}) > \frac{D}{4} > \Delta^{n_1}\,.
   \end{equation}
   Fix also some $u_m \in C_m \cap U$.
   We will establish that $\dist_{\omega_{\bm{P}}}(u_1, u_m)$ is large, certifying
   that $\diam_{\omega_{\bm{P}}}(U)$ is large as well.

   Let $N_{i} \seteq |(U \cap C_{i}) \setminus C_{i+1}|$ for $i=0,1,\ldots,m.$
   Note that $N_i \geq 1$ from \eqref{eq:N0}.
   Define
   \[
         \ell_i \seteq \levelt_{\bm{P}}(C_i,n_i) \quad \textrm{ for } i \in \{0,1,\ldots,m\}\,,
   \]
   and observe from \eqref{eq:thelevels} that
   \begin{equation}\label{eq:N1}
      2^{\ell_i} \geq N_{i}/2\,.
   \end{equation}
   And by construction,
   \begin{equation}\label{eq:N2}
      \sum_{i=0}^{m} N_{i} = |U| = 2^k\,.
   \end{equation}

   \begin{figure}
      \begin{center}
            \includegraphics[width=9cm]{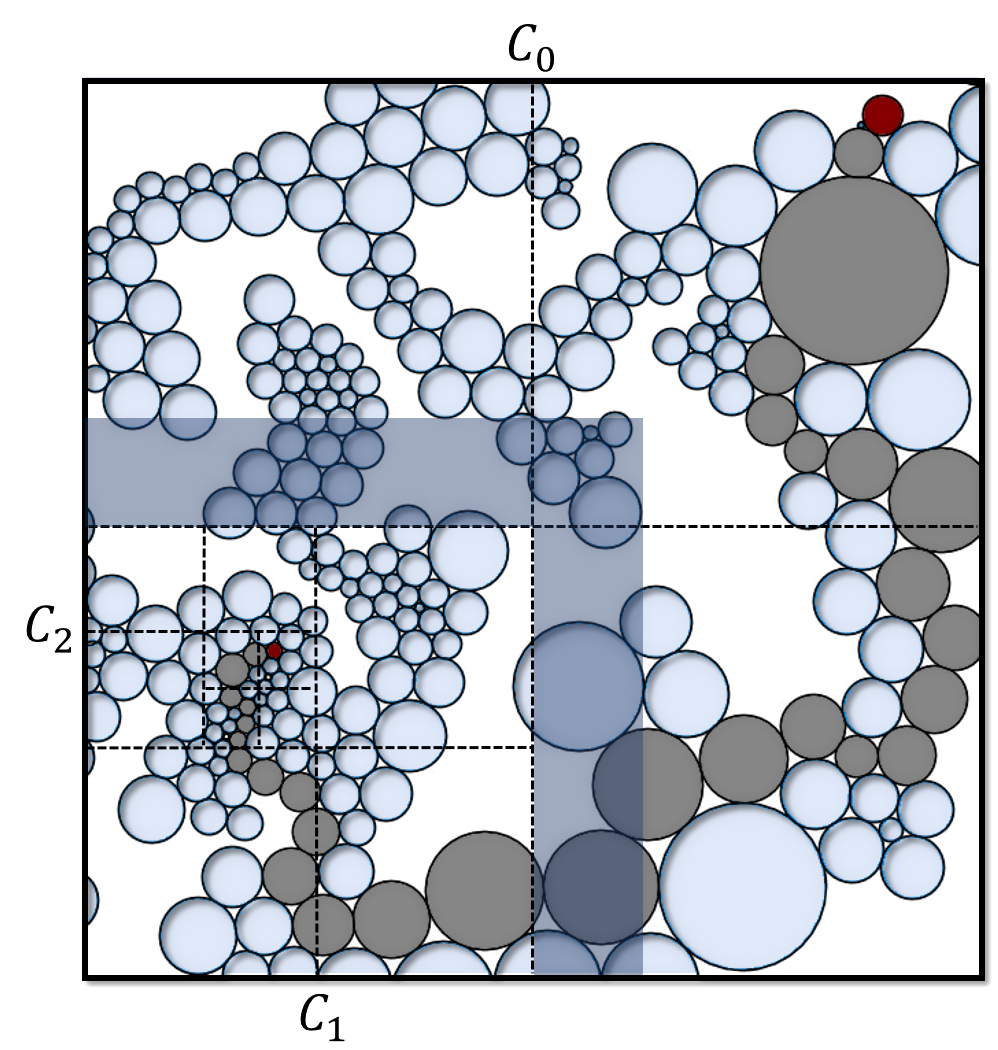}
            \caption{The path $\gamma$ from $u_0 \in C_0$ to $u_m \in C_m$ passing through $N(C_1,\Delta^{n_1})\setminus C_1$.
            \label{fig:buffer}}
      \end{center}
      \end{figure}

   \medskip
   \noindent
   {\bf The length of a $u_0$-$u_m$ path.}
   Let $\gamma = \langle v_0, v_1, v_2, \ldots, v_t\rangle$ be an arbitrary
   simple path in $G$ with $v_0=u_0$ and $v_t=u_m$.  Our goal is to prove that
   \begin{equation}\label{eq:goal}
      \len_{\omega_{\bm{P}}}(\gamma) \gtrsim 2^{k/d_*}\,,
   \end{equation}
   since if this holds for all such paths $\gamma$, it verifies \eqref{eq:otherside}.

   The basic outline is as as follows.
   Informally, imagine that $\gamma$ is parameterized by arclength in the metric $\dist$.
   While $\gamma$ need not spend much time in a cube $C_i$,
   it must cross from outside $C_{i-1}$ to inside $C_{i}$, 
   and therefore it must spend time $\asymp \Delta^{n_{i}}$
   in the neighborhood $N(C_i,\Delta^{n_i})$, where its $\dist_{\omega_{\bm{P}}}$-length
   experiences a reweighting by $\theta_{\bm{P}}^{(C_i,n_i)}$.
   See \pref{fig:buffer}.
   We will now split $\gamma$ into subpaths $\gamma_0,\gamma_1,\ldots,\gamma_m$ accordingly
   and show that the reweighting is sufficient to yield \eqref{eq:goal}.

   \medskip

   For $i \in \{1,\ldots,m\}$, 
   let $s_i$ denote the largest index for which $v_{s_i} \in \gamma$
   satisfies $v_{s_i} \notin N(C_i, \Delta^{n_i})$, and let $t_i$ denote
   the smallest index for which $t_i > s_i$ and $v_{t_i} \in N(C_i, \Delta^{n_i}/2)$.
   Such indices must exist because $\gamma$ begins at $u_0 \notin N(C_1,\Delta^{n_1})$ (recall \eqref{eq:large-dist})
   and $\gamma$ ends at $u_m \in C_{m}$.
   Let $\gamma_i$ denote the subpath $\langle v_{s_i}, \ldots, v_{t_i}\rangle$.
   Define $\gamma_0$ similarly 
   unless $\gamma \subseteq N(C_0, \Delta^{n_0})$.
   In that case, we define $\gamma_0 \seteq \gamma$.
   Observe that, by construction,
   \begin{equation}\label{eq:len-dist}
      \len_{\dist}(\gamma_i) \gtrsim \Delta^{n_i}\,.
   \end{equation}
   For $i \geq 1$, this follows from $v_{s_i} \notin N(C_i, \Delta^{n_i})$ but $v_{t_i} \in N(C_i,\Delta^{n_i}/2)$.
   If $i=0$ and $\gamma_0 = \gamma$, it follows from
   \[
      \len_{\dist}(\gamma) \geq \dist(u_0,u_m) \stackrel{\eqref{eq:large-dist}}{\geq} D/4 \gtrsim \Delta^{n_0}\,.
   \]

   This yields a lower bound on the $\omega_0$-length of each $\gamma_i$.

   \begin{lemma}\label{lem:help3}
      For each $i \in \{0,1,\ldots,m\}$,
      \[
         \len_{\omega_0}(\gamma_i) \gtrsim \Delta^{n_i}\,.
      \]
   \end{lemma}

   \begin{proof}
      Parameterize $\gamma_i = \langle x_1, x_2, \ldots, x_h\rangle$.
      From \eqref{eq:set-dist}, we have
      \begin{equation}\label{eq:distvs}
         \dist(x_j, x_{j+1})\lesssim \diam(S_{x_j}) + \diam(S_{x_{j+1}})\lesssim \omega_0(x_j) + \omega_0(x_{j+1})\,,
      \end{equation}
      where the last inequality is \eqref{eq:omega-compare}.

   We conclude that
   \[
      \len_{\omega_0}(\gamma_i) \geq \frac12 \sum_{j=1}^h \omega_0(x_j) \stackrel{\eqref{eq:distvs}}{\gtrsim} 
      \sum_{j=1}^{h-1} \dist(x_j,x_{j+1}) =
      \len_{\dist}(\gamma_i) \gtrsim \Delta^{n_i}\,.\qedhere
   \]
\end{proof}

   Toward proving \eqref{eq:goal}, observe that
   \begin{equation}\label{eq:bnd1}
      \len_{\omega_{{\bm{P}}}}(\gamma) \geq \frac12 \sum_{j=0}^{t} \omega_{{\bm{P}}}(v_j) \geq \frac12 \sum_{j=0}^{t} \omega_0(v_j)
      \left(\sum_{i=0}^{m} \left(\theta^{(C_i,n_i)}_{{\bm{P}}}(v_j)\right)^{d_*} \right)^{1/d_*}
   \end{equation}

   Recall that $1\leq K_0 \lesssim 1$ was chosen so that $\diam(S_v) \leq K_0 \omega_0(v)$ for all $v \in V$.
	Recalling \eqref{eq:set-dist}, let $1 \leq K_1 \lesssim 1$ be such that
\[
	\max \left\{ \dist(x,y) : x\in S_u, y \in S_v\right\} \leq K_1 \left(\diam(S_u)+\diam(S_v)\right)\qquad \forall \{u,v\} \in E\,.
\]
   For each $v \in V$, denote
   \[
      L(v) \seteq \left\{ i \in \{0,1,\ldots,m\} : \omega_0(v) > \frac{\Delta^{n_i}}{8 K_0 K_1} \textrm{ and } S_v \cap N(C_i, 2\tau \Delta^{n_i}) \neq \emptyset\right\}\,.
   \]
   This is the set of indices $i$ such that $S_v$ intersects the neighborhood of $C_i$ but $\diam(S_v)$ is
   ``large'' with respect to $\diam(C_i)$.

   Define the subset
   \[
      \Lambda \seteq \left\{ i \in \{0,1,\ldots,m\} : i \notin \bigcup_{v \in \gamma} L(v) \right\}\,,
   \]
   and the quantities
   \begin{align*}
      N_{\Lambda} &\seteq \sum_{i \in \Lambda} N_i  \\
      N_{\bar{\Lambda}} &\seteq 2^k - N_{\Lambda}\,.
   \end{align*}
   Clearly the following two claims suffice to establish \eqref{eq:goal}.

   \begin{lemma}[Large bodies]\label{lem:elf1}
      If $N_{\bar{\Lambda}} \geq 2^{k-1}$, then
      \[
         \len_{\omega_{\bm{P}}}(\gamma) \gtrsim 2^{k/d_*}\,.
      \]
   \end{lemma}

   \begin{lemma}[Small bodies]\label{lem:elf2}
      If $N_{\Lambda} \geq 2^{k-1}$, then
      \[
         \len_{\omega_{\bm{P}}}(\gamma) \gtrsim 2^{k/d_*}\,.
      \]
   \end{lemma}

   In proving these two lemmas, we will need the following elementary estimate.
   It is a discretized version of the fact that the $x \mapsto (\log x)^{-2/d_*} x^{1/d_*}$ is concave
   on the interval $[c,\infty)$ for some $c > 1$.

   \begin{lemma}\label{lem:concave}
      For some integer $A \geq 2$,
      consider $S_A = \{ (a_0, a_1, \ldots, a_k) \in \Z_+^{k+1} : A = a_0 2^{k} + a_1 2^{k-1} + \cdots + a_k\}$
      Then the quantity
      \begin{equation}\label{eq:obj}
         \sum_{i=0}^k a_i \frac{2^{(k-i)/d_*}}{(7+i)^{2/d_*}}
      \end{equation}
      is minimized over $S_A$ when $a_1, \ldots, a_k \in \{0,1\}$.
   \end{lemma}

   \begin{proof}
      Consider any $(a_0, a_1, \ldots, a_k) \in S_A$ such that $a_i \geq 2$ for some $i > 0$.
      Then $(a'_0, a'_1, \ldots, a'_k) \in S_A$ where $a_j' = a_j$ if $j \notin \{i,i-1\}$,
      and $a_i' = a_i-2, a_{i-1}' = a_{i-1}+1$.
      We can calculate the change in the value of \eqref{eq:obj}:
      \begin{align*}
         \frac{2^{(k-i)/d_*}}{(6+i)^{2/d_*}} - 2 \frac{2^{(k-(i+1))/d_*}}{(7+i)^{2/d_*}}
         &= 2^{(k-i)/d_*} \left(\frac{1}{(6+i)^{2/d_*}} - \frac{2^{1-1/d_*}}{(7+i)^{2/d_*}}\right) \\
         &= \frac{2^{(k-i)/d_*}}{(7+i)^{2/d_*}} \left(\left(1+\frac{1}{6+i}\right)^{2/d_*}-2^{1-1/d_*}\right) < 0\,,
      \end{align*}
      where we have used $d_* \geq 2$.
   \end{proof}

   \begin{corollary}\label{cor:helper5}
      Suppose that for some $a_0, a_1, a_2, \ldots, a_k \in \Z_+$, it holds that
      $a_0 2^k + a_1 2^{k-1} + \cdots + a_k \geq 2^{k-2}$.
      Then,
      \[
         \sum_{i=0}^k a_i \frac{2^{(k-i)/d_*}}{(1+i)^{2/d_*}} \geq \frac{2^{(k-2)/d_*}}{9}\,.
      \]
   \end{corollary}

   \begin{proof}
      Applying \pref{lem:concave} gives
      \[
         \sum_{i=0}^k a_i \frac{2^{(k-i)/d_*}}{(1+i)^{2/d_*}} \geq 
         \sum_{i=0}^k a_i \frac{2^{(k-i)/d_*}}{(7+i)^{2/d_*}} \geq \frac{2^{(k-2)/d_*}}{9^{2/d_*}} \geq \frac{2^{(k-2)/d_*}}{9}\,.\qedhere
      \]
   \end{proof}

\medskip
\noindent
{\bf Contribution from large bodies.}
   Now we can prove \pref{lem:elf1}.

   \begin{proof}[Proof of \pref{lem:elf1}]
   From the definition \eqref{eq:theta-def}, we have
   \begin{equation}\label{eq:eqbnd1}
      i \in L(v) \implies \left(\omega_0(v) \theta^{(C_i,n_i)}_{\bm{P}}(v)\right)^{d_*} \gtrsim \frac{2^{\ell_i}}{(1+k-\ell_i)^2}
   \end{equation}
      Using \eqref{eq:bnd1} in conjunction with \eqref{eq:eqbnd1} yields
      \begin{equation}\label{eq:f11}
         \len_{\omega_{\bm{P}}}(\gamma)
            \gtrsim \sum_{v \in \gamma} \left(\sum_{i \in L(v)} \frac{2^{\ell_i}}{(1+k-\ell_i)^2}\right)^{1/d_*}
         \geq \sum_{v \in \gamma} \sum_{i \in L(v)} \frac{2^{\ell_i/d_*}}{(1+k-\ell_i)^{2/d_*}}\,.
      \end{equation}
      Now from \eqref{eq:N1}, we have
      \[
         \sum_{v \in \gamma} \sum_{i \in L(v)} 2^{\ell_i} \geq N_{\bar{\Lambda}}/2 \geq 2^{k-2}\,.
      \]
      Thus \pref{cor:helper5} in conjunction with \eqref{eq:f11} yields the desired bound.
   \end{proof}

   \medskip
   \noindent
   {\bf Contribution from small bodies.}
   Once we restrict ourselves to subpaths $\gamma_i$ composed of bodies that are ``small''
   with respect to the scale of the cube $C_i$, we can argue that the corresponding
   subpaths are well-behaved.

   \begin{lemma}\label{lem:help1}
      For every $i \in \Lambda$, if $\gamma_i = \langle x_1, \ldots, x_h\rangle$, then 
      \[
         \dist(x_j, x_{j+1}) \leq \frac{\Delta^{n_i}}{4} \quad \textrm{for}\quad j=1,2,\ldots,h-1\,.
      \]
      In particular, it holds that
      $\gamma_i \subseteq N(C_i, 2 \Delta^{n_i}) \setminus N(C_i, \Delta^{n_i}/4)$.
   \end{lemma}
   
   \begin{proof}
      By construction, we have
      $x_2, \ldots, x_h \in N(C_i, \Delta^{n_i})$ and $x_1, \ldots, x_{h-1} \notin N(C_i, \Delta^{n_i}/2)$.
      Thus the second assertion of the lemma follows from the first.

      To verify the former, note that
      since $x_2, \ldots, x_h \in N(C_i, \Delta^{n_i})$, we have $S_{x_j} \cap N(C_i, \Delta^{n_i}) \neq \emptyset$ for $j=2,\ldots,h$.
      Therefore since $i \in \Lambda$, it holds that
      $\omega_0(x_j) \leq \frac{\Delta^{n_i}}{8 K_0 K_1}$ for $j=2,\ldots,h$.
      In particular, $\diam(S_{x_2}) \leq K_0 \omega_0(x_2) \leq \frac{\Delta^{n_i}}{8}$.
      Since $\{x_1,x_2\} \in E$, the quasi-tangency condition \eqref{eq:coarse-tan} gives
      \[
         \dist(S_{x_1},S_{x_2}) \leq \tau\cdot \diam(S_{x_2}) \leq \tau \frac{\Delta^{n_i}}{8}\,,
      \]
      and therefore
      \[
         S_{x_2} \cap N(C_i, \Delta^{n_i}) \neq \emptyset \implies S_{x_1} \cap N(C_i, 2\tau\Delta^{n_i}) \neq \emptyset\,.
      \]
      Since $i \in \Lambda$, we have $\omega_0(x_1) \leq \frac{\Delta^{n_i}}{8 K_0 K_1}$ as well.

      Using this in conjuction with \eqref{eq:set-dist}, it holds that
      for $j=1,2,\ldots,h-1$, since $\{x_j,x_{j+1}\} \in E(G)$,
      \[
         \dist(x_j, x_{j+1}) \leq K_1\left(\diam(S_{x_j}) + \diam(S_{x_{j+1}})\right) \leq K_0 K_1 (\omega_0(x_j)+\omega_0(x_{j+1})) \leq 2 K_0 K_1 \frac{\Delta^{n_i}}{8 K_0 K_1} \leq \frac{\Delta^{n_i}}{4}\,.\qedhere
      \]
   \end{proof}

   Recall that $\gamma = \langle v_0, v_1, \ldots, v_t\rangle$.

   \begin{lemma}\label{lem:help2}
      For each $j \in \{0,1,\ldots,t\}$, $v_j$ occurs in at most one subpath $\{\gamma_i : i\in \Lambda\}$.
   \end{lemma}

   \begin{proof}
      Note that since $n_{i+1} \leq n_i - s$ for all $i=0,1,\ldots,m-1$, and $\Delta \geq 2, s\geq 4$,
      the sets $N(C_i, 2 \Delta^{n_i}) \setminus N(C_i, \Delta^{n_i}/4)$ are pairwise disjoint
      for all $i=0,1,\ldots,m$.  Hence the result follows from \pref{lem:help1}.
   \end{proof}

   We can now finish the proof.

   \begin{proof}[Proof of \pref{lem:elf2}]
      First, note that \pref{lem:help2} implies that for every $j\in\{0,1,\ldots,t\}$,
      \[
         \left(\sum_{\substack{i \in \Lambda :\\ v_j \in \gamma_i}} 
   \left(\theta^{(C_i,n_i)}_{{\bm{P}}}(v_j)\right)^{d_*} \right)^{1/d_*}
   =
      \sum_{\substack{i \in \Lambda :\\ v_j \in \gamma_i}}\theta^{(C_i,n_i)}_{{\bm{P}}}(v_j)\,.
      \]
   Using this in \eqref{eq:bnd1} yields
   \begin{equation}\label{eq:elfa}
      \len_{\omega_{\bm{P}}}(\gamma)
      \geq \frac12 \sum_{j=0}^t \omega_0(v_j) \sum_{\substack{i \in \Lambda :\\ v_j \in \gamma_i}}\theta^{(C_i,n_i)}_{{\bm{P}}}(v_j) 
      =\frac12
      \sum_{i \in \Lambda} \left(\sum_{v \in \gamma_i} \theta_{\bm{P}}^{(C_i,n_i)}(v) \omega_0(v)\right)\,.
   \end{equation}

      From \pref{lem:help3}, we know that 
      \begin{equation}\label{eq:elf21}
         \sum_{v \in \gamma_i} \omega_0(v) \gtrsim \Delta^{n_i}\,.
      \end{equation}

      For $i \in \Lambda$, \pref{lem:help1} yields $\gamma_i \subseteq N(C_i, 2 \Delta^{n_i})$, hence
      $S_{v} \cap N(C_i, 2 \Delta^{n_i}) \neq \emptyset$ for each $v \in \gamma_i$.
      From the definition of $\Lambda$, this yields $\omega_0(v) \leq \frac{\Delta^{n_i}}{8 K_0}$, thus
      from the definition \eqref{eq:theta-def},
      \[
         v \in \gamma_i \implies \theta_{\bm{P}}^{(C_i,n_i)}(v) \gtrsim \Delta^{-n_i} \frac{2^{\ell_i/d_*}}{(1+k-\ell_i)^{2/d_*}}\,.
      \]

      Combining this with \eqref{eq:elfa} and \eqref{eq:elf21} gives
      \begin{equation}\label{eq:soclose}
         \len_{\omega_{\bm{P}}}(\gamma) \gtrsim \sum_{i \in \Lambda} \frac{2^{\ell_i/d_*}}{(1+k-\ell_i)^{2/d}}\,.
      \end{equation}
      By \eqref{eq:N1} and our assumption that $N_{\Lambda} = \sum_{i \in \Lambda} N_i \geq 2^{k-1}$,
      we have $\sum_{i \in \Lambda} 2^{\ell_i} \geq 2^{k-2}$. 
      Thus \pref{cor:helper5} in conjunction with \eqref{eq:soclose} yields
      \[
         \len_{\omega_{\bm{P}}}(\gamma) \gtrsim 2^{k/d_*}\,,
      \]
   completing the proof.\qedhere
\end{proof}
\end{proof}

\subsection{$d$-parabolicity}
\label{sec:parabolic}

We first discuss two examples showing that for distributional limits of finite graphs
with uniformly bounded degrees, $d$-parabolicity and the property that $\dimconfover^d(G,\rho) \leq d$
are incomparable.

First, we remark on the following general construction.  Let $\{(H_n,\rho_n) : n \geq 1\}$ be 
a sequence of non-isomorphic, finite rooted graphs, and let $p$ be a probability on $\N$.
Let $(\bm{H},\bm{h})$ be the random rooted graph that arises by choosing $(H_n,\rho_n)$ with probability $p(n)$.
Suppose furthermore that
\begin{equation}\label{eq:isuni}
   \E \left[|V(\bm{H})|\right] = \sum_{n \geq 1} p(n) |V(H_n)| < \infty\,.
\end{equation}

Consider a path $P_N$ of length $N \geq 1$, and attach to each vertex of $P_N$ an independent copy
of $(\bm{H},\bm{h})$ (we identify $\bm{h}$ with the corresponding vertex in $P_N$).
This yields a random graph $G_N$, and we choose a root $r_N \in V(G_N)$ uniformly at random.
We claim that $\{(G_N,r_N)\}$ has a distributional limit $(G,\rho)$.
To see this, note that
\[
   q(n) \seteq \lim_{N \to \infty} \Pr[r_N \textrm{ is in a copy of $H_n$}] = \frac{p(n) |V(H_n)|}{\E [V(\bm{H})]}\,.
\]
Now \eqref{eq:isuni} implies that $q$ is a probability on $\N$.

It is then straightforward to describe the limit:  $(G,\rho)$ is a bi-infinite path $P$ with
some fixed vertex $v_0 \in V(P)$.  At $v_0$, we attach a copy $H$ of $(H_n,\rho_n)$ with
probability $q(n)$, and choose $\rho \in V(H)$ uniformly at random.  At every vertex in $V(P) \setminus \{v_0\}$,
we attach an independent copy of $(\bm{H},\bm{h})$.

Using the weight $W(v) \seteq \frac{\1_{V(P)}(v)}{1+\dist_G(v_0,v)}$ verifies
the following claim.

\begin{claim}\label{claim:Gpara}
   $G$ is almost surely $2$-parabolic.
\end{claim}

\begin{example}[Infinite conformal growth exponent but $2$-parabolic]
Now
let $\{H_n : n \geq 1\}$ denote an infinite family of connected, transitive, $d$-regular graphs with
$|V(H_n)| \in [n,2n]$ and
\begin{equation}\label{eq:thediam}
   \diam(H_n) < C \log (n+1)\,,
\end{equation}
for some $C > 0$.  (The diameter here refers to the graph metric.)
For instance, one can take a family of expanding Cayley graphs.

\begin{lemma}
If $\rho_n \in V(H_n)$ is uniformly random, then for any $\omega : V(H_n) \to \R_+$:
\begin{equation}\label{eq:fast-growth}
   \max_{x \in V(H_n)} \left|B_{\omega}\left(x, 2C \log(n+1) \sqrt{\E[\omega(\rho_n)^2]}\right)\right| \geq \frac{n}{4}\,.
\end{equation}
\end{lemma}

\begin{proof}
   Consider the following family of convex sets indexed by $D > 0$:
   \[
      \cC_D \seteq \left\{ \omega : \E[\omega(\rho_n)^2] \leq 1 \textrm{ and } \frac{1}{|V(H_n)|^2} \sum_{x,y \in V(H_n)}
   \dist_{\omega}(x,y) \geq D \right\}.
\]
By convexity and transitivity of $H_n$, $\omega_0 \in \cC_D \iff \cC_D \neq \emptyset$, where
$\omega_0 \equiv 1$ is the uniform weight.
Note that $\dist_{\omega_0}$
is simply the graph metric $\dist_{H_n}$,
hence \eqref{eq:thediam} implies that $\cC_{C \log (n+1)} = \emptyset$.

Thus for any $\omega : V(H_n) \to \R_+$, there is an $x_0 \in V(H_n)$ such that
\[
   \frac{1}{|V(H_n)|} \sum_{x \in V(H_n)} \dist_{\omega}(x,x_0) < C \sqrt{\E[\omega(\rho_n)^2]}\cdot \log(n+1)\,.
\]
In particular, for $R \seteq C \sqrt{\E[\omega(\rho_n)^2]}\cdot \log(n+1)$,
it holds that
\[
   \left|B_{\omega}(x_0, 2R)\right| > \tfrac12 |V(H_n)|\,,
\]
completing the proof.
\end{proof}

Define $p(n) \seteq \frac{c'}{n^2 (\log (n+1))^2}$, where the constant $c'$ is chosen so that $p$ is a 
probability on $\N$.  Then \eqref{eq:isuni} is satisfied, hence there
is a distributional limit $(G,\rho)$ as above.  By \pref{claim:Gpara},
$G$ is almost surely $2$-parabolic.

Let $\omega$ denote a (unimodular) $L^2$-normalized conformal weight on $(G,\rho)$, and define the numbers
\[
W_n \seteq \sqrt{\E\left[\omega(\rho)^2 \mid \rho \textrm{ is in a copy of $H_n$}\right]}\,.
\]
Since $\omega$ is $L^2$-normalized, we have
\[
   \sum_{n \geq 1} q(n) W_n^2 \leq 1\,.
\]
Because $q(n) \asymp \frac{1}{n (\log n)^2}$, there
must exist an infinite set $I \subseteq \N$ such that
$n \in I \implies W_n \leq \log n$.

Note that the Mass-Transport Principle yields, for $n \geq 2$,
\[
   \E\left[\frac{1}{|V(H)|} \sum_{x \in V(H)} \omega(x)^2 \bigmid \textrm{$\rho$ is in a copy $H$ of $H_n$}\right]
   = W_n^2\,,
\]
hence Markov's inequality gives
\[
   \Pr\left[\frac{1}{|V(H)|} \sum_{x \in V(H)} \omega(x)^2 > (\log n)^2 W_n^2 \bigmid \textrm{$\rho$ is in a copy $H$ of $H_n$}\right] \leq \frac{1}{\log n}\,.
\]
Applying the Mass-Transport Principle again, a straightforward application of Borel-Cantelli
shows that almost surely there are infinitely many $n \in I$ such that $G$
contains a copy $H$ of $H_n$ with
\[
   \frac{1}{|V(H)|} \sum_{x \in V(H)} \omega(x)^2 < (\log n)^2 W_n^2 \leq (\log n)^4\,.
\]
And in this case, \eqref{eq:fast-growth} yields
\[
   \max_{v \in V(H)} \left|B_{\omega}(v, 2 C \log(n+1)^3 \right| \geq \frac{n}{4}\,,
\]
clearly ruling out any finite growth exponent.
This demonstrates that $\dimconfunder(G,\rho) = \infty$.
\end{example}

\begin{example}[$2$-dimensional conformal growth, but not $2$-parabolic]
   We will exhibit a unimodular random graph $(\widehat{T},\rho)$ with
   $\deg_{\widehat{T}}(\rho) \leq 6$ almost surely, and such that
   $\widehat{T}$ is almost surely transient (and hence {\em not} $2$-parabolic), yet
   $\dimconfover(\widehat{T},\rho) \leq 2$.

   Denote by $T_n$ the complete $4$-ary tree of height $n \geq 1$.
   Let us obtain a graph $\widetilde{T}_n$ by
   replacing every edge at distance $h$ from the leaves
   by $f(h) $ parallel paths of length $g(h)$, with
   \begin{align*}
      f(h) &\seteq 2^{h}\,, \\
      g(h) &\seteq \left\lceil 2^{h-\sqrt{h}}\right\rceil\,.
   \end{align*}

   Observe that for any $x \in V(\widetilde{T}_n)$ and $i \geq 0$, it holds that
   \begin{equation}\label{eq:subq}
      |B_{\widetilde{T}_n}(x, 2^{i-\sqrt{i}})| \leq O(1) \sum_{j=1}^i 4^{i-j} f(j) g(j) \leq O(4^i)\,,
   \end{equation}
   and moreover there is a flow from a leaf of $\widetilde{T}_n$ to the root with energy at most
   \begin{equation}\label{eq:enbnd}
      O(1) \sum_{j=1}^h \frac{g(j)}{f(j)} \leq O(1)\,.
   \end{equation}

   Thus if we let $(T,\rho)$ denote the distributional limit of $\{(\widetilde{T}_n, \rho_n)\}$
   with $\rho_n \in V(\widetilde{T}_n)$ chosen uniformly at random, then
   \eqref{eq:enbnd} implies that $T$ is almost surely transient, and $\eqref{eq:subq}$ implies
   that $\dimconfover(T,\rho) \leq 2$ (using the normalized conformal weight $\omega \equiv 1$).

   The only remaining issue is that the vertex degrees in $(T,\rho)$ are not bounded.
   Since every distributional limit of finite planar graphs
   with uniformly bounded degrees {\em is} $2$-parabolic,
   replacing the parallel paths with bounded-degree subgraphs
   will require the final step in our construction to be non-planar.

   To obtain uniformly bounded degrees,
   we replace every vertex $x \in V(T_n)$ at distance $h=0,1,2,\ldots$ from the
   leaves with a cloud $C_x$ containing $f(h) = 2^h$ vertices.
   Moreover, if $y \in V(T_n)$ is a child of $x$, we connect
   every vertex in $C_y$ to exactly two vertices of $C_x$ via internally-disjoint paths of length $g(h)$
   to obtain a graph $\widehat{T}_n$.

   Clearly one can do this in a manner so that if $x$ is an internal node of $T_n$, then
   the degree of every vertex in $C_x$ in $\widehat{T}_n$ is precisely $6$ (one path from each of its
   four children and two paths to its parent), unless $x$ is the root of $T_n$,
   in which case the vertices in $C_x$ have degree $4$.
   Now let $(\widehat{T},\rho)$ denote the distributional limit of $\{(\widehat{T}_n, \rho_n)\}$
   where $\rho_n \in V(\widehat{T}_n)$ is chosen uniformly at random.

   It is straightforward that both the growth and energy estimates \eqref{eq:subq} and \eqref{eq:enbnd}
   hold for $\widehat{T}_n$ as well, where now the flow is from a leaf to the cloud $C_r$ of the root $r \in V(\widehat{T}_n)$.
   Therefore $(\widehat{T},\rho)$ is a unimodular
   random graph with essentially bounded degrees that is almost surely transient
   (and hence {\em not} $2$-parabolic) but which satisfies
   $\dimconfover(\widehat{T},\rho)\leq 2$.

   Using the duality between $d$-parabolicity and the $\ell^{d'}$ energy of a flow to $\infty$ (where $d'=\frac{d}{d-1}$ is the dual exponent to $d$), one can similarly construct examples, for every $d \geq 2$,
   of unimodular random graphs $(G,\rho)$ such that
   is almost surely not $d$-parabolic but satisfies $\dimconfover^d(G,\rho) \leq d$.
\end{example}

\subsubsection{Gauged conformal growth and vertex extremal length}
\label{sec:gauged-extremal}

We now prove that gauged $d$-dimensional conformal growth implies
$d$-parabolicity when
the degree of the root is almost surely uniformly bounded.

\begin{proof}[Proof of \pref{thm:scaled-strong}]
   Fix $d \geq 1$ and a unimodular random graph $(G,\rho)$ with gauged $d$-dimensional conformal growth
   and such that $\deg_G(\rho)$ is essentially bounded.
   For each $R \geq 0$, let $\omega_R$ be an $L^d$-normalized conformal metric on $(G,\rho)$ that satisfies
   \begin{equation}\label{eq:onescale}
      \|B_{\omega_R}(\rho,R)\|_{L^{\infty}} \leq C R^d
   \end{equation}
   for some constant $C \geq 1$.

   From \cite[Lem. 2.6]{Lee17a}, we may assume that for each $R \geq 0$, the following additional properties hold almost surely:
   \begin{enumerate}
      \item For all $x \in V(G)$,  $\omega_R(x) \geq 1/2$.
      \item For all $\{x,y\} \in E(G)$, we have $\omega_R(x) \leq C' \omega_R(y)$,
            where $C' > 1$ is a constant depending only on $\|\deg_G(\rho)\|_{L^{\infty}}$.
   \end{enumerate}
   Moreover, these additional properties are sufficient to guarantee
   that we can compare $\dist_{\omega_R}$ balls to $\dist_G$ balls
   in the following sense (see \cite[Lem. 2.5]{Lee17a}):
   Almost surely, for every $x \in V(G)$ and $R, r \geq 0$,
   \begin{equation}\label{eq:ball-compare}
      B_G\left(x,\frac{\log \frac{r}{2\omega_R(x)}}{\log C'}\right) \subseteq B_{\omega_R}(x,r) \subseteq B_G(x,2r)\,.
   \end{equation}

   Fix $\e \in (0,1), n \geq 1$.
   Let $\{r_{j}\}$ be the sequence of numbers with $r_1=1$ and, that satisfies, for $j > 1$,
   \[
      \frac{\log \frac{\e r_j}{16 C'}}{\log C'} = 2 r_{j-1}\,.
   \]
   Denote
   \[
      \Lambda_G \seteq \left\{ x \in V(G) : \omega_{r_j}(x) \leq \frac{1}{\e} \textrm{ for } j \leq n \right\}\,.
   \]
   For $x \in V(G)$, let
   \[
      A_j(x) \seteq B_{\omega_{r_j}}(x,r_j) \setminus B_{\omega_{r_{j}}}\left(x,\frac{r_{j}}{8C'}\right)\,.
   \]
   By our choice of the sequence $\{r_j\}$ and \eqref{eq:ball-compare}, for every $x \in \Lambda_G$, we have
   \begin{equation}\label{eq:bc2}
      B_{\omega_{r_{j-1}}}(x, r_{j-1}) \subseteq B_G(x, 2 r^{j-1}) \subseteq B_{\omega_{r_j}}(x, r_j/(8C'))\,,
   \end{equation}
   hence if $x \in \Lambda_G$, then the sets $A_1(x), A_2(x), \ldots, A_n(x)$ are pairwise disjoint.

   Consider now the following conformal weight which depends on the choice of some $z \in V(G)$:
   \[
      \omega_{(z)}(x) \seteq \left(\sum_{j=1}^n r_j^{-d} \omega_{r_j}(x)^d \1_{A_j(z)}(x)\right)^{1/d}\,.
   \]
   By construction, if $z \in \Lambda_G$, then
   \begin{equation}\label{eq:vbnd1}
      \sum_{x \in V(G)} \omega_{(z)}(x)^d \leq \sum_{j=0}^n r_j^{-d} \cV_{\omega_{r_j}}(z,r_j)\,,
   \end{equation}
   where
   \[
      \cV_{\omega}(x,r) \seteq \sum_{y \in B_{\omega}(x,r)} \omega(y)^d\,,
   \]
   we used the fact established earlier that $z \in \Lambda_G$ implies that the sets $A_j(z)$ are pairwise disjoint for $j=1,2,\ldots,n$.

   Now observe that
   \begin{equation}\label{eq:observe1}
      \dist_{\omega_{(z)}}(z, x) \geq \sum_{j=1}^n \frac{\dist_{\omega_{r_j} \1_{A_j(z)}}(z,x)}{r_j}\,.
   \end{equation}
   Suppose that $x \in V(G) \setminus B_G(z, 2r_n)$ and
   consider any path $\gamma$ from $z$ to $x$ in $G$.
   Let $\gamma_j$ denote the portion of $\gamma$ which lies inside $A_{j}(z)$.
   Every vertex $u \in B_{\omega_j}(z,r_j/(8C'))$ satisfies $\omega_{r_j}(u) \leq r_j/(4C')$
   by definition of $\dist_{\omega_{r_j}}$, thus if $\{u,v\} \in E(G)$,
   then by Property (2) above, $\omega(v) \leq r_j/4$.
   
   In particular,
   \[
      \len_{\omega_{r_j} \1_{A_j(z)}}(\gamma) = \len_{\omega_{r_j}}(\gamma_j) \geq \frac{r_j}{2C'}\,.
   \]
   Using \eqref{eq:observe1}, we conclude that
   \begin{equation}\label{eq:dist-bnd}
      z \in \Lambda_G \implies
      \dist_{\omega_{(z)}}(z, V(G) \setminus B_G(z, 2 r_n)) \geq
      \sum_{j=1}^n \frac{r_j}{2C' r_j} \geq \frac{n}{2C'}\,.
   \end{equation}

   Let us now return to \eqref{eq:vbnd1}.  
   For a conformal metric $\omega : V(G) \to \R_+$ and some $R > 0$,
   define the transport
   \[
      F(G,\omega,x,y) = \omega(x)^d \1_{\{\dist_{\omega}(x,y) \leq R\}}\,.
   \]
   Then by the Mass-Transport Principle,
   \begin{align*}
      \E\left[\cV_{\omega}(\rho,R)\right] = \E\left[\sum_{x \in V(G)} F(G,\omega,x,\rho)\right] &= \E\left[\sum_{x \in V(G)} F(G,\omega,\rho,x)\right] \\
      &= \E\left[\omega(\rho)^d |B_{\omega}(\rho,R)|\right] \leq \|B_{\omega}(\rho,R)\|_{L^{\infty}} \E\left[\omega(\rho)^d\right]\,.
   \end{align*}
   We conclude from \eqref{eq:onescale} that for each $j \leq n$,
   \[
      \E\left[\cV_{\omega_{r_j}}(\rho,r_j) \mid \rho \in \Lambda_G\right] \leq \frac{C r_j^d}{\Pr[\rho \in \Lambda_G]}\,,
   \]
   hence
   \[
      \E \left[\sum_{x \in V(G)} \omega_{(\rho)}(x)^d \mid \rho \in \Lambda_G\right] \leq \frac{C n}{1-\e^d n}\,,
   \]
   where we have used Markov's inequality and a union bound to assert that
   $\Pr[\rho \in \Lambda_G] \geq 1-\e^d n$.

   Take $\e = 1/n$ and $n \geq 2$ in the preceding construction and define the event
   \[
      \cE(n) \seteq \left\{ \omega_{r_j}(\rho) \leq n \textrm{ for } j \leq n \textrm{ and }
         \|\omega_{(\rho)}\|^d_{\ell^d(V(G)} \leq  2 C n^{1.5}\right\}\,.
   \]
   By Markov's inequality and a union bound, we have
   \[
      \Pr\left(\cE(n)\right) \geq 1-\frac{2}{\sqrt{n}}\,.
   \]
   Moreover from \eqref{eq:dist-bnd},
   \[
      \cE(n) \implies \frac{\dist_{\omega_{(\rho)}}\left(\rho, V(G) \setminus B_G(\rho, 2 r_n)\right)}{\|\omega_{(\rho)}\|_{\ell_d(V(G))}} \geq
      \frac{n}{4C'C^{1/d} n^{1.5/d}} \geq \frac{n^{1/4}}{4C'\sqrt{C}}\,.
   \]

   In other words, for every $n \geq 1$, it holds that
   \[
      \Pr\left[\vel_d(\Gamma_G(\rho)) \geq \frac{n^{1/4}}{4C'\sqrt{C}}\right] \geq 1-\frac{2}{\sqrt{n}}\,.
   \]
   Sending $n \to \infty$, it follows that
   \[
      \Pr\left[\vel_d(\Gamma_G(\rho))=\infty\right]  = 1\,,
   \]
   i.e., almost surely $G$ is $d$-parabolic.
\end{proof}

\subsection{Spectral bounds for the graph Laplacian}
\label{sec:spectral}

We now prove the following generalization of \pref{thm:spectral-intro}.

\begin{theorem}\label{thm:spectral}
   For every $d,\tau,M \geq 1$, $c_1,c_2 > 0$, there is a constant $C \geq 1$ such that the following holds.
   Suppose $G=(V,E)$ is an $n$-vertex graph that is $(\tau,M)$-quasi-packed in a $(c_1,c_2,d)$-regular space $(X,\dist,\mu)$.
   Then for $k=1,2,\ldots,n-1$,
   \[
      \lambda_k(G) \leq 
      C \frac{\Delta_G(k)}{k} \left(\log \frac{n}{k}\right)^2 \left(\frac{k}{n}\right)^{2/d}\,.
   \]
\end{theorem}

Consider a finite connected graph $G=(V,E)$.
Define the {\em Rayleigh quotient $\cR_G(f)$} of non-zero $f : V \to \R$ by
\[
\cR_G(f) \seteq 
\frac{\sum_{\{x,y\} \in E} |f(x)-f(y)|^2}{\sum_{x \in V} \deg_G(x) f(x)^2}\,.
\]
It is an elementary fact (see, e.g., \cite[Cor. 3.1]{Lee17a}) that to establish \pref{thm:spectral}, it suffices
to find $k$ disjointly supported functions $\f_1, \f_2, \ldots, \f_k : V \to \R$ such that
for each $i=1,2,\ldots,k$,
\[
   \cR_G(\f_i) \leq 
      C \frac{\Delta_G(k)}{k} \left(\log \frac{n}{k}\right)^2 \left(\frac{k}{n}\right)^{2/d}\,.
\]

Toward this end, we now state \cite[Thm. 3.12]{Lee17a}.
For a finite graph $G=(V,E)$, denote
$\avgd_G(\e) \seteq \frac{\Delta_G(\e |V|)}{\e |V|}$.

\begin{theorem}\label{thm:bumps}
   There is a constant $C \geq 1$ such that the following holds.
   Consider a finite graph $G=(V,E)$ with $n=|V|$.
   Suppose that $\omega : V \to \R_+$ is a conformal metric on $G$ satisfying
   \begin{enumerate}
      \item $\frac{1}{|V|} \sum_{x \in V} \omega(x)^2 \leq 1$\,,
      \item
         For some numbers $R > 0, K \geq 2$:
   \begin{equation}\label{eq:growth-require}
      \max_{x \in V} |B_{\omega}(x,R)| \leq K \leq n/2\,.
   \end{equation}
   \end{enumerate}
   Then there exist disjoint supported functions $\f_1, \f_2, \ldots, \f_k : V \to \R_+$
   with $k \geq n/16K$, and such that
   \[
      \max \left\{\cR_G(\f_1), \ldots, \cR_G(\f_k)\right\} \leq C
      \frac{(\log K)^2 \left(\vphantom{\bigoplus}\avgd_G(1/K)+\avgd_G\left(1/R^2\right)\right)}{R^2}\,.
   \]
\end{theorem}

\begin{remark}
   The statement of \cite[Thm. 3.12]{Lee17a} contains an additional parameter $\alpha$,
   and here we have used the fact that one can take $\alpha \leq O(\log K)$.
   This is a basic and well-known estimate; it follows, for instance, from \cite[Lem. 4.5]{Lee17a}
   which it itself a reference to \cite[Lem. 3.11]{LN05}.
\end{remark}

Now \pref{thm:spectral} is a consequence of the following proposition combined with \pref{thm:finite-sphere-packing}.

\begin{proposition}
   Suppose that $G=(V,E)$ is an $n$-vertex graph with $(c,R,d)$-growth for some numbers $c \geq 1, d \geq 2$ and
   all $R \geq 0$.  
   Then for $k=1,2,\ldots,n-1$,
   \[
      \lambda_k(G) \leq O(1) \frac{\Delta_G(k)}{k} \left(\log \frac{n}{k}\right)^2 \left(\frac{c k}{n}\right)^{2/d}\,.
   \]
\end{proposition}

\begin{proof}
   For each $R \geq 0$, let $\omega_R : V \to \R_+$ be a conformal metric on $G$ satisfying
   \[
      \frac{1}{|V|} \sum_{x \in V} \omega_R(x)^d = 1\,,
   \]
   and
   \[
      \max_{x \in V} |B_{\omega}(x,R)| \leq c R^d\,.
   \]
   Note that from H\"older's inequality,
   \[
      \frac{1}{|V|} \sum_{x \in V} \omega_R(x)^2 \leq \left(\frac{1}{|V|} \sum_{x \in V} \omega_R(x)^d\right)^{2/d} = 1\,.
   \]
   So we can apply \pref{thm:bumps} with $\omega_R$ and $K=cR^d$ to obtain, for $k \leq n/(16 cR^d)$,
   \[
      \lambda_k(G) \leq O(1) \frac{(d \log R)^2 \avgd_G(\frac{1}{cR^d})}{R^2}\,.
   \]
   Setting $R \seteq (n/16 ck)^{1/d}$ yields
   \[
      \lambda_k(G) \leq O(1) \left(\frac{c k}{n}\right)^{2/d} \left(\log \frac{n}{k}\right)^2 \frac{\Delta_G(k)}{k}\,.
   \]
   completing the proof.
\end{proof}

\subsection*{Acknowledgements}

I am grateful to Omer Angel, Itai Benjamini, and Asaf Nachmias
for enlightening discussions about distributional limits of graphs,
and to Pierre Pansu for many insightful conversations during a
semester on ``Metric geometry, algorithms, and groups''
at the IHP in 2011.
Thanks are also due to Austin Stromme and the
anonymous referees for a careful reading of earlier drafts.

\bibliographystyle{alpha}
\bibliography{diffusive}

\end{document}